\documentclass[hidelinks,onefignum,onetabnum]{siamart220329}

\usepackage{amssymb}
\usepackage{mathrsfs}
\usepackage{bm}
\usepackage{booktabs}
\usepackage{cite}

\newcommand{\dvg}{\mathrm{div}}
\newcommand{\cond}{\mathrm{cond}}
\newcommand{\eig}{\mathrm{eig}}

\usepackage{lipsum}
\usepackage{amsfonts}
\usepackage{graphicx}
\usepackage{epstopdf}
\usepackage{algorithmic}
\ifpdf
  \DeclareGraphicsExtensions{.eps,.pdf,.png,.jpg}
\else
  \DeclareGraphicsExtensions{.eps}
\fi

\newsiamremark{remark}{Remark}
\newsiamremark{hypothesis}{Hypothesis}
\crefname{hypothesis}{Hypothesis}{Hypotheses}
\newsiamthm{claim}{Claim}

\headers{BDDC/FETI-DP preconditioners for Biot's model}{H. Chu\quad L. F. Pavarino\quad S. Zampini }

\title{Block BDDC/FETI-DP Preconditioners for Three-Field mixed finite element Discretizations of Biot's consolidation model}

\author{ Hanyu Chu 
  \thanks{Dipartimento di Matematica, Università degli Studi di Pavia, Via Ferrata 5, 27100, Pavia, Italy (\email{hanyu.chu@unipv.it}).}
\and Luca Franco Pavarino
\thanks{Dipartimento di Matematica, Università degli Studi di Pavia, Via Ferrata 5, 27100, Pavia, Italy (\email{luca.pavarino@unipv.it}).}
\and Stefano Zampini
\thanks{Computing Research Center, King Abdullah University of Science and Technology, Thuwal 23955, Saudi Arabia (\email{stefano.zampini@kaust.edu.sa}).}
}

\usepackage{amsopn}

\NewDocumentCommand{\dgalext}{m}{%
  \sbox0{%
    \mathsurround=0pt 
    $\left\{\vphantom{#1}\right.\kern-\nulldelimiterspace$%
}%
  \sbox2{\{}%
  \ifdim\ht0=\ht2
    \{\kern-.45\wd2 \{#1\}\kern-.45\wd2 \}%
  \else
  \fi
}

\NewDocumentCommand{\dgalx}{om}{%
  \sbox0{\mathsurround=0pt$#1\{$}%
  \sbox2{\{}%
  \ifdim\ht0=\ht2
    \{\kern-.45\wd2 \{#2\}\kern-.45\wd2 \}%
  \else
    \mathopen{#1\{\kern-.5\wd0 #1\{}
    #2
    \mathclose{#1\}\kern-.5\wd0 #1\}}
  \fi
}
\usepackage{subfigure}

\usepackage{multirow}

\graphicspath{{fig/}{fig/mesh/}}

\begin{document}
\maketitle
\begin{abstract}
In this paper, we construct and analyze a block dual-primal preconditioner for Biot's consolidation model approximated by three-field mixed finite elements based on a  displacement, pressure, and total pressure formulation. The domain is decomposed into nonoverlapping subdomains, and the continuity of the displacement component across the subdomain interface is enforced by introducing a Lagrange multiplier. After eliminating all displacement variables and the independent subdomain interior components of pressure and total pressure, the problem is reduced to a symmetric positive definite linear system for the subdomain interface pressure,
total pressure, and the Lagrange multiplier. This reduced system is solved by a preconditioned conjugate gradient method, with a block dual-primal preconditioner using a Balancing Domain Decomposition by Constraints (BDDC) preconditioner for both the interface total pressure block and the interface pressure blocks, as well as a Finite Element Tearing and Interconnecting – Dual Primal (FETI-DP) preconditioner for the Lagrange multiplier block. By analyzing the conditioning of the preconditioned subsystem associated with the interface pressure and total pressure components, we obtain a condition number bound of the preconditioned system, which is scalable in the number of subdomains, poly-logarithmic in the ratio of subdomain and mesh sizes, and robust with respect to the parameters of the model.
Extensive numerical experiments confirm the theoretical result of the proposed algorithm.
\end{abstract}



\section{Introduction}\label{sec:introduction}
Biot’s consolidation model describes the coupling of mechanical deformations and fluid flow in porous media. Due to the complexity of this model, analytical solutions are rarely available, making its numerical approximation a topic of significant research interest, see, e.g., \cite{LMW2017,Yi2017,AKY2020,RHOA2018,ML1992,ML1994,PW2007I,PW2007II,PW2008,KS2005,KSWK2006,khan2022nonsymmetric}. However, solving Biot’s model numerically poses considerable challenges, as it involves coupled multiphysics interactions between elasticity and porous media flow. In particular, numerical schemes, especially those based on two-field formulations \cite{Yi2017,AKY2020}, often suffer from issues such as elasticity locking and pressure oscillations \cite{HOL2012,RHOA2018}.

To address these difficulties, various reformulations employing three-field or four-field formulations have been proposed \cite{LMW2017,OR2016,YB2017}.
A ``\textit{total pressure}'' was introduced in \cite{OR2016,LMW2017} to develop a three-field formulation for Biot’s model, allowing it to be interpreted as a combination of a generalized Stokes problem and a reaction-diffusion problem. This formulation enables the use of an LBB-stable Stokes finite element method for discretizing the displacement and total pressure variables, while a Lagrange element is used for the solid pressure variable, facilitating the discretization process. Our primary focus in the present paper is to design and analyze an efficient preconditioning algorithm for this three-field formulation.

In both two-field and three-field approaches, solving Biot’s model involves handling a large, ill-conditioned, and often indefinite system of linear equations.
As problem size increases, direct solvers become inefficient, making iterative solvers a more viable alternative.
Many preconditioning techniques have been proposed  to improve convergence of Biot’s iterative solvers. For the two-field formulation, a modified Jacobi preconditioner based on the symmetric quasi-minimal residual method was introduced in \cite{2001A}. Block triangular preconditioners were developed in \cite{AGHRZ2018}, while diagonal precondtioners leveraging the Schur complement were proposed in \cite{PTCL2002,TPC2004}.
Building on Schur complement techniques, a block-diagonal preconditioner for a four-field formulation was studied in \cite{BLY2017}, while a preconditioner for a three-field formulation incorporating displacement, fluid velocity, and pressure was studied in \cite{ABB2012}. Alternative robust block preconditioners for the same formulation were proposed in \cite{HK2018,HKLP2019}, while \cite{CHXY2020} established a general preconditioning framework for both two-field and three-field Biot's models.  \cite{LMW2017} studied block preconditioning techniques for a three-field formulation and \cite{GCL2023,CGLM2023} developed and analyzed iterative decoupled algorithms for the same formulation.

This paper develops and analyzes a novel dual-primal block preconditioner based on BDDC and FETI-DP blocks for the parameter-robust three-field formulation introduced in \cite{LMW2017}. In recent years, significant progress has been made in applying FETI-DP methods to the Stokes problem \cite{TL2013,LT2013,TL2015} and the linear elasticity model \cite{WZSP2021}. Building on \cite{WZSP2021}, we make two key contributions to the use of dual-primal methods for preconditioning the three-field formulation of the Biot model.
First, we analyze the reduced Schur complement system associated with the interface and prove that it is symmetric and positive definite. Based on this result, we design a block preconditioner that enables the efficient application of the preconditioned conjugate gradient method.
Second, for the Lagrange multiplier variable, which relaxes the continuity of the displacement variable (as in \cite{TL2013}), we propose two types of preconditioners: lumped-type and Dirichlet-type. For the pressure and total pressure variables, we develop a block BDDC (Balancing Domain Decomposition by Constraints) preconditioner \cite{D2003} for the coupled Schur complement system. By establishing norm equivalence and assuming an appropriate set of primal variables, we show that this BDDC preconditioner effectively preconditions the system, ensuring a scalability bound on the convergence rate.

The remainder of the paper is structured as follows.
Section 2 introduces the Biot model and its three-field finite element discretization. Section 3 presents a dual-primal decomposition, leading to a reduced Biot's system. In Section 4, we propose the block BDDC/FETI-DP preconditioner. Section 5 provides a convergence rate analysis, and Section 6 concludes with numerical results that validate our theoretical findings.

\section{Mixed finite element discretization}\label{sec:model}
Let $\Omega\subset\mathbb{R}^d\,(d=2\text{ or } 3)$ be a polygonal/polyhedral domain, decomposed into $N$ nonoverlapping subdomains $\Omega_i$ of diameter $H_i$, and forming a coarse finite element partition $\mathcal T_H$ of 
$\bar{\Omega}=\cup_{i=1}^N\bar\Omega_i.$
The interface of this decomposition 
is given by $\Gamma=(\cup_{i=1}^N\partial\Omega_i)\setminus\partial\Omega$. In the next subsection, we will further divide each subdomain into several shape-regular finite elements, assuming that the nodes are aligned across the interface between the subdomains.

The governing equations of Biot's consolidation model are given by
\begin{equation}\label{2-filed:biot}
\left\{
\begin{aligned}
\dvg(2\mu\epsilon(\mathbf{u})+\,\lambda\,\dvg\,\mathbf{u}\mathbf{I})+\,\alpha\,\nabla p\,&=\mathbf{f}  \quad \text{in } \Omega, \\
c_0\,p+\alpha\,\dvg\,\mathbf{u}- \dvg\,\kappa\nabla p\,&=g  \quad \text{in } \Omega.
\end{aligned}
\right.
\end{equation}
Here, the unknowns are the displacement vector of the solid phase $\mathbf{u}$, the pressure of the fluid phase $p$. The coefficient $\alpha$ is the Biot-Willis constant which is close to one, $\epsilon(\mathbf u)$ is the symmetric gradient of $\mathbf u$, $\mathbf{f}$ is the body force, $\kappa$ is the hydraulic conductivity, and $g$ is the source term. The storage coefficient $c_0$ represents the increase in fluid volume for a unit increase in fluid pressure while maintaining constant volumetric strain.
The Lam\'e constants $\lambda>0$ and $\mu>0$, here assumed to be constant
$\lambda_i, \mu_i$
within each subdomain $\Omega_i$, are given by
$\lambda_i=\frac{E_i\nu_i}{(1+\nu_i)(1-2\nu_i)},\quad \mu_i=\frac{E_i}{2(1+\nu_i)},$
with $E_i$ being the Young's modulus and $\nu_i$ being the Poisson's ratio. In many practical applications,
$c_0$ is typically of the order
 $c_0\sim\frac1\lambda$. To simplify the parameters in our system, we will set $c_0 = \frac{\alpha^2}{\lambda}$ for the rest of the discussion in this paper.

To ensure the existence and uniqueness of the solution, we impose a combination of partial Neumann and partial Dirichlet boundary conditions. Specifically, we assume that the boundary is partitioned as $\partial \Omega =\bm{\Gamma}_d \cup \bm{\Gamma}_t=\bm{\Gamma}_p \cup \bm{\Gamma}_f$,
with $|\bm\Gamma_d|> 0$,  $|\bm{\Gamma}_p|>0$, where $\bm\Gamma_d$, $\bm{\Gamma}_p$ are the Dirichlet boundaries for $\mathbf{u}$ and $p$, and $\bm\Gamma_t$, $\bm{\Gamma}_f$ are the Neumann boundaries for $\mathbf{u}$ and $p$, respectively. For instance
\begin{align}
    \begin{aligned}
        \mathbf{u} &=0,\quad \text{on}\, \bm\Gamma_d, &
         p&=0,\quad \ \text{on }\,\bm\Gamma_p,\\
(2\mu\epsilon(\mathbf{u})+\,\lambda\,\dvg\,\mathbf{u}\mathbf{I}-\alpha p\bm I)\bm{n}&=\bm{h},\quad \text{on}\,\bm\Gamma_t,
        & \quad \quad \kappa\nabla p\cdot\bm{n}&=g_2,\quad \text{on}\,\bm\Gamma_f,
    \end{aligned}
\end{align}
where $\bm n$ is the unit outward normal to the boundary. Without loss of generality, we assume that the Dirichlet boundary conditions mentioned above are homogeneous.

Following \cite{LMW2017}, we introduce the so-called total pressure
$\xi=-\lambda\,\dvg\,\mathbf{u}+\alpha\,p$, so that \eqref{2-filed:biot} can be rewritten as a three-field Biot's system
\begin{equation}\label{model problem}
\left\{
\begin{aligned}
-2\mu\,\dvg\,\epsilon(\mathbf u)+\nabla\xi&=\mathbf{f} \quad \text{in } \Omega,\\
-\dvg\,\mathbf{u}-\lambda^{-1}(\xi-\alpha p)&=0 \quad \text{in } \Omega,\\
\lambda^{-1}\,(\alpha\xi-2\alpha^2p)+\dvg\,(\kappa\nabla p)&=g \quad \text{in } \Omega.
\end{aligned}
\right.
\end{equation}
The weak form of \eqref{model problem} is: find $\mathbf{u}\in (H^1_{\bm\Gamma_d}(\Omega))^d,\,\xi\in L^2(\Omega),\,p\in H^1_{\bm\Gamma_p}(\Omega)$ such that
\begin{equation}\label{variation problem}
\left\{
\begin{aligned}
a(\mathbf{u},\,\mathbf{v})+b(\mathbf{v},\,\xi)&=(\mathbf{f},\,\mathbf{v}) &&\forall \mathbf{v}\in (H^1_{\bm\Gamma_d}(\Omega))^d, \\
b(\mathbf{u},\,\xi)-c(\xi,\,\eta)+d(\,p,\,\eta)&=0 &&\forall \eta\in L^2(\Omega),\\
d(\xi,\,q)-e(p,\,q)&=(g,\,q) &&\forall q\in H^1_{\bm\Gamma_p}(\Omega),
\end{aligned}
\right.
\end{equation}
where
\begin{align}
&a(\mathbf{u},\,\mathbf{v})=\sum_{i=1}^Na_i(\mathbf{u},\,\mathbf{v})=2\sum_{i=1}^N\mu_i\int_{\Omega_i}\epsilon(\mathbf{u}):\epsilon(\mathbf{v}) {\rm d}x,\\
&b(\mathbf{v},\,\xi)=\sum_{i=1}^Nb_i(\mathbf{v},\,\xi)=\sum_{i=1}^N\int_{\Omega_i}-(\dvg\,\mathbf{v})\xi {\rm d}x,\\
&c(\xi,\eta)=\sum_{i=1}^Nc_i(\xi,\eta)=\frac1\lambda\int_{\Omega_i}\xi\,\eta{\rm d}x,\quad
d(p,\,\eta)=\sum_{i=1}^Nd_i(p,\,\eta)=\frac\alpha\lambda\int_{\Omega_i}p\,\eta {\rm d}x,\\
&e(p,\,q)=\sum_{i=1}^Ne_i(p,\,q)=\int_{\Omega_i}\Big(\kappa\nabla p\cdot\nabla q+\frac{2\alpha^2}{\lambda}p\,q\Big) {\rm d}x.
\end{align}

A mixed finite element method is employed to solve problem \eqref{variation problem}. Let $\mathcal T_h$ be a family of conforming, quasi-uniform, and regular triangulations of the domain $\Omega$, parametrized by $h$.
Denote the displacement finite element space by $\mathbf{V}\subset(H^1_{\bm\Gamma_d}(\Omega))^d$, the total pressure finite element space by $W\subset L^2(\Omega)$ and the pressure finite element space by $Q\subset H^1_{\bm\Gamma_p}(\Omega)$. The product space is denoted by $\mathcal X_h = \mathbf V\times W\times Q$.
The finite element approximation $(\mathbf{u}_h,\,\xi_h,\,p_h)\in\mathcal X_h$ of the variational problem \eqref{variation problem} satisfies
\begin{equation}\label{eq:dczP}
\left\{
\begin{aligned}
a(\mathbf{u}_h,\,\mathbf{v}_h)+b(\mathbf{v}_h,\,\xi_h)&=(\mathbf{f},\,\mathbf{v}) &&\forall \mathbf{v}_h\in \mathbf{V}, \\
b(\mathbf{u}_h,\,\xi_h)-c(\xi_h,\,\eta_h)+d(\,p_h,\,\eta_h)&=0 &&\forall \eta\in W,\\
d(\xi_h,\,q_h)-e(p_h,\,q_h)&=(g,\,q_h) &&\forall q_h\in Q.
\end{aligned}
\right.
\end{equation}
System \eqref{eq:dczP} can be restated in matrix form as
\begin{equation}\label{eq:maxtirxform}
\left[
    \begin{array}{ccc}
    A&B^T&0\\
    B&-C&D^T\\
    0&D&-E
    \end{array}
    \right]
    \left[
    \begin{array}{c}
    \mathbf{u}\\ \xi\\ p
    \end{array}
    \right]=
    \left[
    \begin{array}{c}
    \mathbf{f}\\ 0\\g
    \end{array}
    \right],
\end{equation}
where $A,\,B,\,C,\,D,\,E,$ represent, respectively, the restrictions of $a(\cdot,\cdot), \dots, e(\cdot,\cdot)$ to the finite dimensional spaces $(\mathbf{V},\,W,\,Q)$. We still denote by $\mathbf f$ and $g$ the restrictions of $(\mathbf{f},\,\cdot)$ and $(g,\cdot)$ to the finite-dimensional spaces $\mathbf{V}$ and $Q$.

In order to ensure the stability of system \eqref{eq:dczP}, we constrain the pair $(\mathbf V, W)$ to be a stable Stokes pair, which means that the following inf-sup condition
\begin{equation}\label{ineq:Stokeslbb}
  \inf_{\eta\in W}\sup_{\mathbf v\in\mathbf V}\frac{(\dvg\mathbf v,\eta)}{\|\mathbf v\|_{\mathbf V}\|\eta\|_W}\geq \beta_S
\end{equation}
holds with a positive constant $\beta_S$. A typical example is the Taylor-Hood element, e.g., \cite{BF1991}.
Based on \eqref{ineq:Stokeslbb},
Lee et al. proved in \cite{LMW2017} that there exists a constant $\beta>0$, independent of $\lambda$, $\alpha$, and $\kappa$, such that
 \begin{equation}\label{ineq:LBB}
  \inf_{(\mathbf u,\xi,p)\in\mathcal X_h}\sup_{(\mathbf v,\eta,q)\in\mathcal X_h}
  \frac{(\mathcal A_h(\mathbf u,\xi,p),(\mathbf v,\eta,q))_{(\mathcal X_h^*,\mathcal X_h)}}
  {\|(\mathbf u,\xi,p)\|_{\mathcal X_h}\|(\mathbf v,\eta,q)\|_{\mathcal X_h}}
  \geq \beta.
\end{equation}
Here the product norm $\|\cdot\|_{\mathcal X_h}$ is defined by
\begin{equation*}
     \|(\mathbf v,\eta,q)\|_{\mathcal X_h}:=
    \Big(\|\epsilon(\mathbf v)\|_0^2+\|\eta\|_0^2+
    \|\alpha\lambda^{-\frac12}p\|_0^2+\|\kappa\nabla p\|_0^2\Big)^{\frac12}.
\end{equation*}
The classical Babuška–Brezzi theory (cf. \cite{BF1991}) guarantees the unique solvability of the discrete problem \eqref{eq:dczP}.

\section{ Subspace decomposition and operators associated to the interface}\label{sec:DD}
\subsection{Subspace decomposition}\label{subsec:subspaces}
The discrete displacement, pressure, and total pressure finite element spaces $\mathbf V$, $Q$, and $W$, are decomposed as
\[\mathbf V=\mathbf V_I\bigoplus\widehat{\mathbf V}_\Gamma,\quad W=W_I\bigoplus\widehat W_\Gamma,\quad\text{ and }\quad Q=Q_I\bigoplus\widehat Q_\Gamma.\]
Here, $\mathbf V_I$, $W_I$ and $Q_I$ are direct sums of independent subdomain interior displacement spaces $\mathbf V_I^{(i)}$, interior total pressure spaces $W_I^{(i)}$ and interior pressure spaces $Q_I^{(i)}$
\[\mathbf V_I=\bigoplus_{i=1}^N\mathbf V_I^{(i)},\quad W_I=\bigoplus_{i=1}^NW_I^{(i)},\quad Q_I=\bigoplus_{i=1}^NQ_I^{(i)},\]
while $\widehat{\mathbf V}_\Gamma$, $\widehat W_\Gamma$, and $\widehat Q_\Gamma$ are subdomain interface spaces associated to displacement, total pressure, and pressure, respectively.
The functions in $\widehat{\mathbf V}_\Gamma$, $\widehat W_\Gamma$ and $\widehat Q_\Gamma$ are continuous across the subdomain boundary, and their degrees of freedom are shared by neighboring subdomains.
Similarly to $\mathbf V_I, W_I$ and $Q_I$, we define the product interface spaces
\[\mathbf V_\Gamma=\bigoplus_{i=1}^N\mathbf V_\Gamma^{(i)},\quad W_\Gamma=\bigoplus_{i=1}^NW_\Gamma^{(i)},\quad Q_\Gamma=\bigoplus_{i=1}^NQ_\Gamma^{(i)},\]
with interface displacement spaces $\mathbf V_\Gamma^{(i)}$, interface total pressure spaces $W_\Gamma^{(i)}$, and interface pressure spaces $Q_\Gamma^{(i)}$.

In order to formulate our dual-primal domain decomposition algorithm, we will also need the intermediate subspaces $\widetilde{\mathbf V}_\Gamma$ and $\widetilde Q_\Gamma$ defined by further splitting the
interface degrees of freedom into \textit{primal} (subscript $\Pi$) and \textit{dual} (subscript $\Delta$) degrees of freedom. More precisely, we employ partially sub-assembled
subdomain interface displacement and pressure finite element spaces:
\[\widetilde{\mathbf V}_\Gamma:=\mathbf V_\Delta\oplus\widehat{\mathbf V}_\Pi=\left(\oplus_{i=1}^N\mathbf V_\Delta^{(i)}\right)\oplus\widehat{\mathbf V}_\Pi, \quad\quad
\widetilde Q_\Gamma:=Q_\Delta\oplus\widehat Q_\Pi=\left(\oplus_{i=1}^NQ_\Delta^{(i)}\right)\oplus\widehat Q_\Pi,\]
respectively. Here, $\widehat{\mathbf V}_\Pi$ is a global subspace consisting of selected continuous, coarse level, \textit{primal} displacement variables.
It is a recognized fact that, for linear elastic problems in nonoverlapping domain decomposition algorithms, the coarse space $\widehat{\mathbf V}_\Pi$ must be sufficiently large to control the rigid body motions, thus guaranteeing a scalable convergence rate, see \cite{KW2006} for more details.
Conversely, an excessive number of primal variables can potentially impede algorithmic efficiency. Thus, minimizing the size of the coarse-level problem remains crucial.

The coarse space $\widehat{\mathbf V}_\Pi$ can generally be spanned by subdomain vertex basis functions and/or edge/face basis functions with constant values at the nodes of the corresponding edge/face. In the two-dimensional case, $\widehat{\mathbf V}_\Pi$ is spanned solely by the subdomain vertex displacement variables. In the three-dimensional case, $\widehat{\mathbf V}_\Pi$ includes all subdomain vertex displacement nodal basis functions and additional edge and/or face average basis functions with constant values. The complementary dual space $\mathbf V_\Delta$, defined by the direct sum of independent subdomain dual interface displacement $\mathbf V_\Delta^{(i)}$, is spanned by the \textit{dual} interface functions that
vanish at the primal degrees of freedom. The definitions of $\widetilde Q_\Gamma$ is analogous.

\subsection{Dual-Primal decomposition}
Let us focus on the space $\mathbf{V}$ of displacements and introduce the associated operator.
Generally, the functions in $\mathbf V_\Delta$ may not be continuous. To enforce continuity, we employ the Boolean matrix
\[B_\Delta:=\left[
\begin{array}{llll}
  B_\Delta^{(1)} &   B_\Delta^{(2)}& \cdots&B_\Delta^{(N)}
\end{array}\right],\]
which can be constructed from the set $\{-1,0,1\}$. Each row of  $B_\Delta$ contains exactly two nonzero entries, $1$ and $-1$, and enforces the continuity of the displacement degrees of freedom among neighboring subdomains. For any $\mathbf v_\Delta$ in $\mathbf V_\Delta$, the condition $B_\Delta \mathbf v_\Delta = 0$ implies that these degrees of freedom across the subdomains must remain consistent. 
Let $\Lambda$ denote the space of Lagrange multipliers, which is the range of $B_\Delta$ applied to $\mathbf V_\Delta$.

Utilizing the decomposition of the space introduced above, we can rephrase the discrete problem \eqref{eq:dczP} by splitting the variables into interior, dual, and primal variables and obtain the following equivalent problem:\\ Find $(\mathbf{u}_I,
\xi_I,p_I,\mathbf{u}_{\Delta}, \mathbf{u}_{\Pi},\xi_{\Gamma}, p_{\Gamma},\lambda_{\Delta}) \in
\mathbf V_I\bigoplus W_I\bigoplus Q_I\bigoplus\mathbf V_\Delta\bigoplus\mathbf V_\Pi\bigoplus \widehat W_\Gamma\bigoplus\widehat Q_\Gamma\bigoplus\Lambda$ such that
\begin{equation}\label{eq:split}
    \left[
    \begin{array}{cccccccc}
A_{II}&B_{II}^T&0&A_{I\Delta}&A_{I\Pi}& B_{\Gamma I}^T&0 & \\
B_{II}&-C_{II}&D_{II}^T&B_{I\Delta}&B_{I\Pi}&-C_{I\Gamma}&D_{\Gamma I}^T& \\
0&D_{II}&-E_{II}&0 &0 &D_{I\Gamma}&-E_{I\Gamma}& \\
A_{\Delta I}&B_{I\Delta}^T&0&A_{\Delta\Delta}&A_{\Delta\Pi}&B_{\Gamma\Delta}^T&0&B_{\Delta}^T\\
A_{\Pi I}&B_{I\Pi}&0&A_{\Pi\Delta}&A_{\Pi\Pi}&B_{\Gamma\Pi}&0& \\
B_{\Gamma I}&-C_{\Gamma I}&D_{I\Gamma}^T&B_{\Gamma\Delta}&B_{\Gamma\Pi}&-C_{\Gamma\Gamma}&D_{\Gamma\Gamma}^T& \\
0&D_{\Gamma I}&-E_{\Gamma I}&0&0&D_{\Gamma\Gamma}&-E_{\Gamma\Gamma}& \\
& & & B_{\Delta} & & & &
  \end{array}
    \right]
    \left[
    \begin{array}{c}
\mathbf{u}_I\\
\xi_I\\p_I\\
\mathbf{u}_{\Delta}\\ \mathbf{u}_{\Pi}\\
\xi_{\Gamma}\\ p_{\Gamma}\\
\lambda_{\Delta}\\
      \end{array}
    \right]
    =
    \left[
    \begin{array}{c}
    \mathbf{f}_I\\
    0\\ g_I\\ \mathbf{f}_{\Delta}\\
    \mathbf{f}_{\Pi}\\0\\g_{\Gamma}\\0
    \end{array}
    \right],
\end{equation}
where the subblocks within the coefficient matrix denote the restrictions of respective operators in \eqref{eq:maxtirxform} to suitable subspaces.

\subsection{Reduced system of linear equations}\label{sec:reduced}
System \eqref{eq:split} can be reduced to a Schur complement problem for the variables $(\xi_{\Gamma},\,p_{\Gamma},\,\lambda_{\Delta})$, since the leading five-by-five block of the coefficient matrix in \eqref{eq:split} is invertible, 
obtainining:
 \begin{equation}\label{eq:SchurComplement}
   G
   \left[ \begin{array}{c}
   \xi_\Gamma\\  p_\Gamma\\   \lambda_\Delta
    \end{array}\right]
   =g,
 \end{equation}
 where
\[G=\widetilde C+\widetilde B_C\widetilde A^{-1}\widetilde B_C^T,\quad g=\widetilde B_C\widetilde A^{-1} f,\]
with
\[
\widetilde A =
 \left[
    \begin{array}{ccccc}
A_{II}&B_{II}^T&0&A_{I\Delta}&A_{I\Pi}\\
B_{II}&-C_{II}&D_{II}^T&B_{I\Delta}&B_{I\Pi} \\
0&D_{II}&-E_{II}&0 &0  \\
A_{\Delta I}&B_{I\Delta}^T&0&A_{\Delta\Delta}&A_{\Delta\Pi}\\
A_{\Pi I}&B_{I\Pi}&0&A_{\Pi\Delta}&A_{\Pi\Pi}
  \end{array}
\right], \quad
f=
\left[
    \begin{array}{c}
    \mathbf{f}_I\\
    0\\ g_I\\ \mathbf{f}_{\Delta}\\
    \mathbf{f}_{\Pi}
    \end{array}
    \right],
\]
\[
\widetilde B_C=
\left[
    \begin{array}{ccccc}
B_{\Gamma I}&-C_{\Gamma I}&D_{I\Gamma}^T&B_{\Gamma\Delta}&B_{\Gamma\Pi} \\
0&D_{\Gamma I}&-E_{\Gamma I}&0&0 \\
0&0 &0 & B_{\Delta} &0
  \end{array}
    \right],\quad
    \widetilde C=\left[
    \begin{array}{ccc}
    -C_{\Gamma\Gamma}&D_{\Gamma\Gamma}^T&0\\
    D_{\Gamma\Gamma}&-E_{\Gamma\Gamma}&0\\
    0&0&0
    \end{array}
    \right].
\]

The main computation of multiplying $G$ by a vector and the construction of the right-hand side term $g$ in equation \eqref{eq:SchurComplement} are based on the calculation of the action of $\widetilde A^{-1}$.

We define the Schur complement operator at the coarse level
\[S_{\Pi\Pi}:=A_{\Pi\Pi}-A_{\Pi r}A_{rr}^{-1}A_{r\Pi},\]
where
\[A_{rr}:=\left[
    \begin{array}{cccc}
A_{II}&B_{II}^T&0&A_{I\Delta}\\
B_{II}&-C_{II}&D_{II}^T&B_{I\Delta} \\
0&D_{II}&-E_{II}&0  \\
A_{\Delta I}&B_{I\Delta}^T&0&A_{\Delta\Delta}
  \end{array}
\right],~~
A_{\Pi r}=A_{r\Pi}^T:=\left[
    \begin{array}{cccc}
A_{\Pi I}&B_{I\Pi}&0&A_{\Pi\Delta}
  \end{array}
\right].
\]
Then, for any given vector
$z=\left[z_r\,\,z_\Pi\right]^T$, we can compute the product $\widetilde A^{-1}z$ as
\[\widetilde A^{-1}z=\left[\begin{array}{c}A_{rr}^{-1}z_r\\0 \end{array}\right]+
\left[\begin{array}{c}-A_{rr}^{-1}A_{\Pi r}^TS_{\Pi\Pi}^{-1}(z_\Pi-A_{\Pi r}A_{rr}^{-1}z_r)\\
S_{\Pi\Pi}^{-1}(z_\Pi-A_{\Pi r}A_{rr}^{-1}z_r) \end{array}\right].\]
The actions of $A_{rr}^{-1}$ on $z_r$ and on $A_{\Pi r}^TS_{\Pi\Pi}^{-1}(z_\Pi-A_{\Pi r}A_{rr}^{-1}z_r)$ require us to solve a set of independent subdomain saddle problems with Neumann boundary conditions for the
dual variables, while the action of $S_{\Pi\Pi}^{-1}$ on $z_\Pi-A_{\Pi r}A_{rr}^{-1}z_r$ requires one to solve a coarse level problem.

\subsection{Preliminary results}
 \begin{lemma}\label{lem:bound2pressures}
 For any $\eta\in W, q\in Q$, there exists a positive constant $c$ such that
 \begin{equation}\label{ineq:bound2pressures}
     (q^TD\eta)^2\leq c(\eta^T C\eta+q^TEq).
 \end{equation}
 \end{lemma}
This is a direct consequence of the Cauchy-Schwarz inequality.
 \begin{lemma}\label{lem:SPD2pressures}
 For any $\eta\in W, q\in Q$, we have
 \begin{equation}\label{ineq:SPD2pressures}
   \frac{3-\sqrt5}{2}\Big(\eta^T C\eta+q^TEq\Big)\leq
   \left[\begin{array}{cc}
         \eta^T& q^T
      \end{array}\right]
      \left[
      \begin{array}{cc}
        C & -D^T \\
        -D & E
      \end{array}\right]\left[
      \begin{array}{c}
       \eta \\ q
      \end{array}\right].
 \end{equation}
 \end{lemma}
 \begin{proof}
By the definition of matrices $C, D, E$, the right-hand size of \eqref{ineq:SPD2pressures} equals
  \begin{equation}\label{eq:SPD2pressures1}
    \lambda^{-1}\|\eta\|_0^2-2\alpha\lambda^{-1}(\eta,q)+2\alpha^2\lambda^{-1}\|q\|_0^2+(\kappa \nabla q,\nabla q).
 \end{equation}
   An application of Young's inequality leads us to
 \begin{equation}\label{ineq:SPD2pressures2}
    -2\alpha\lambda^{-1}(\eta,q)\geq-\frac{\sqrt5-1}{2} \lambda^{-1}\|\eta\|_0^2-\frac{2}{\sqrt5-1}\alpha^2\lambda^{-1}\|q\|_0^2.
 \end{equation}
 Therefore, we arrive at
 \begin{equation}\label{ineq:SPD2pressures3}
 \begin{split}
    &\lambda^{-1}\|\eta\|_0^2-2\alpha\lambda^{-1}(\eta,q)+2\alpha^2\lambda^{-1}\|q\|_0^2+(\kappa \nabla q,\nabla q)\\
    \geq& \frac{3-\sqrt5}{2}\Big(\lambda^{-1}\|\eta\|_0^2+\alpha^2\lambda^{-1}\|q\|_0^2\Big)+(\kappa \nabla q,\nabla q)
    \geq \frac{3-\sqrt5}{2}\Big(\eta^T C\eta+q^TEq\Big),
    \end{split}
 \end{equation}
 which is the desired result.
 \end{proof}
Consequently, from Lemma \ref{lem:SPD2pressures} and the Sylvester law of inertia, we can claim that $G$ is a symmetric positive definite matrix. Moreover, we recall the following Lemma, the proof of which can be found in \cite{WZSP2021}; see also \cite{GPW2003}.
\begin{lemma}\label{lem:apriori}
Let $\left[\begin{array}{c}\bf v\\ \eta\\q
\end{array}\right]$ satisfy
$\left[
    \begin{array}{ccc}
    A&B^T&0\\
    B&-C&D^T\\
    0&D&-E
    \end{array}
    \right]\left[\begin{array}{c}\bf v\\ \eta\\q
\end{array}\right]=\left[\begin{array}{c}\bf f\\ g\\h
\end{array}\right]$. Then 
\begin{equation}\label{ineq:apriori}
\begin{split}
 & \mathbf v^TA\mathbf v+\left[
 \eta^T \ q^T
 \right]\left[
      \begin{array}{cc}
       \! C & -D^T\! \\
       \! -D & E\!
      \end{array}\right]
      \left[\begin{array}{c}\eta\\ q\end{array}\right]
      \leq\,\,  c\Big(\mathbf f^TA^{-1}\mathbf f+
      g^T(\lambda C)^{-1}g+h^TE^{-1}h\Big).
      \end{split}
\end{equation}
\end{lemma}

\section{Subspace solvers and operators}\label{sec:Preconditioner}

To develop our theory, we shall consider the block BDDC/FETI-DP preconditioner
\begin{equation}\label{eq:blockPreconditioner}
  M^{-1}=\left[
  \begin{array}{ccc}
  M_{\xi_\Gamma}^{-1} && \\ 
    &M_{p_\Gamma}^{-1}& \\
   &&M_{\lambda_\Delta}^{-1}
  \end{array}
  \right]
\end{equation}
for the Schur complement $G$ written in its three-by-three block structure
\begin{equation}\label{eq:splitG}
  G=\left[
  \begin{array}{ccc}
    G_{\xi_\Gamma\xi_\Gamma}&G_{\xi_\Gamma p_\Gamma}&G_{\xi_\Gamma\lambda_\Delta} \\
    G_{p_\Gamma\xi_\Gamma}&G_{p_\Gamma p_\Gamma}&G_{p_\Gamma\lambda_\Delta} \\
   G_{\lambda_\Delta\xi_\Gamma}&G_{\lambda_\Delta p_\Gamma}&G_{\lambda_\Delta\lambda_\Delta}
  \end{array}
  \right].
\end{equation}

\subsection{Restriction and scaling operators}\label{subsec:operators}
Defining specific restriction and interpolation operators is essential to build our preconditioners. These operators are represented by matrices whose entries are elements of the set $\{0,1\}$. We begin by considering the pressure variable and define:
\begin{equation}\label{def:resops}
  \begin{array}{llll}
    &R_{p,\Gamma\Delta}:\widetilde Q_\Gamma\to Q_\Delta,
    &\ \ R_{p,\Gamma\Pi}: \widetilde Q_\Gamma\to \widehat Q_\Pi,
    &\ \  \overline R_{p,\Gamma}:\widetilde Q_\Gamma\to Q_\Gamma, \\
    &R_{p,\Gamma}^{(i)}:\widetilde Q_\Gamma\to Q_\Gamma^{(i)},
    &\ \  R_{p,\Delta}^{(i)}:Q_\Delta\to Q_\Delta^{(i)},
    &\ \  R_{p,\Pi}^{(i)}:\widehat Q_\Pi\to Q_\Pi^{(i)}.
  \end{array}
\end{equation}
Utilizing these operators, we construct the subsequent operators:
\[R_{p,\Gamma}=\bigoplus_{i=1}^N R_{p,\Gamma}^{(i)},\qquad
\widetilde R_{p,\Gamma}=R_{p,\Gamma\Pi}\bigoplus_{i=1}^NR_{p,\Delta}^{(i)}R_{p,\Gamma\Delta}.\]
For each node $x$ on the interface $\Gamma_i:=\partial\Omega_i\cap\Gamma$ of subdomain $\Omega_i$, we define the standard counting function pseudoinverses 
\[\delta_{p,i}^\dagger(x):=\kappa_i(x)/\Bigg(\sum_{j\in\mathcal N_x}\kappa_j(x)\Bigg),\]
where $\mathcal N_x$ is the set of indices of the subdomains having the node $x$ on their boundaries. For simplicity, we assume that
$\kappa_i=\kappa|_{\Omega_i}\text{  and  }\mu_i=\mu|_{\Omega_i}$
are constants on each subdomain $\Omega_i$.
The pseudoinverse functions $\delta_{\xi,i}^\dagger(x)$ and $\delta_{\mathbf u,i}^\dagger(x)$ for total pressure and displacement can be defined similarly, that is
\[\delta_{\xi,i}^\dagger(x):=\mu_i^{-1}(x)/\Bigg(\sum_{j\in\mathcal N_x}\mu_j^{-1}(x)\Bigg),\qquad \delta_{\mathbf u,i}^\dagger(x):=\mu_i(x)/\Bigg(\sum_{j\in\mathcal N_x}\mu_j(x)\Bigg).\]
Multiplying the sole nonzero element of each row of $B_\Delta$ by the scalar $\delta_{\mathbf u,i}^\dagger(x)$, the scaled operator $B_{\Delta,D}$ is defined as follow
\[B_{\Delta,D}:=\left[
\begin{array}{llll}
 D_{\Delta} B_\Delta^{(1)} & D_{\Delta}B_\Delta^{(2)}& \cdots& D_{\Delta}B_\Delta^{(N)}
\end{array}\right],\]
where \(D_\Delta\) is a diagonal matrix containing $\delta_{\mathbf u,i}^\dagger(x)$ on its diagonal.

The scaled local restriction operators $R_{p,D,\Gamma}^{(i)}$ and $R_{p,D,\Delta}^{(i)}$ are defined by multiplying the
sole nonzero element of each row of $R_{p,\Gamma}^{(i)}$ and $R_{p,\Delta}^{(i)}$ by $\delta_{p,i}^\dagger(x)$. Finally, we define
\[R_{p,D,\Gamma}:=\bigoplus_{i=1}^N R_{p,D,\Gamma}^{(i)},\qquad \widetilde R_{p,D,\Gamma}:=R_{p,\Gamma\Pi}\bigoplus_{i=1}^NR_{p,D,\Delta}^{(i)}R_{p,\Gamma\Delta}.\]

Similarly, we must define the restriction and scaling operators for the total pressure variable. We do not need any primal variable for the total pressure variable when designing the preconditioner.
Therefore, for total pressure, we no longer subdivide the degrees of freedom on $\Gamma$ into primal and dual. Finally, we set
\[R_{\xi,\Gamma}:=\bigoplus_{i=1}^N R_{\xi,\Gamma}^{(i)},\quad
R_{\xi,D,\Gamma}:=\bigoplus_{i=1}^N R_{\xi,D,\Gamma}^{(i)},\]

where the restriction operator
$R_{\xi,\Gamma}^{(i)}:\widehat W_\Gamma\to W_\Gamma^{(i)}$
is constructed from the set $\{0,1\}$ and the scaled local restriction operator $R_{\xi,D,\Gamma}^{(i)}$ is obtained by multiplying the sole nonzero element of each row of $R_{\xi,\Gamma}^{(i)}$ by $\delta_{\xi,i}^\dagger(x)$.

\subsection{The total pressure sub-solver}\label{subsec:totalpressure}
We will consider two Schur complements associated with the spaces $W_\Gamma$ and $\widehat W_\Gamma$. For $i=1,\cdots,N$, we define \[\frac{\lambda_i}{\mu_i}S_{\xi,\Gamma}^{(i)}=\frac{\lambda_i}{\mu_i}\Big(C_{\Delta\Delta}^{(i)}-C_{\Delta I}^{(i)}C_{II}^{(i)^{-1}}
C_{I\Delta }^{(i)}\Big),\]
and then define the block Schur complement $S_{\xi,\Gamma}$ corresponding to $W_\Gamma$ by:
\[\frac\lambda\mu S_{\xi,\Gamma}:=\left[\begin{array}{cccc}
\frac{\lambda_1}{\mu_1}S_{\xi,\Gamma}^{(1)}&&&\\
&\frac{\lambda_2}{\mu_2}S_{\xi,\Gamma}^{(2)}&& \\
&&\ddots&\\
&&&\frac{\lambda_N}{\mu_N}S_{\xi,\Gamma}^{(N)}
\end{array}\right].\]
The classical Schur complement $\frac\lambda\mu\widehat S_{\xi,\Gamma}$ defined on the continuous subspace $\widehat W_\Gamma$ is obtained by sub-assembling all the degrees of freedom of the local Schur complements
\[\frac\lambda\mu\widehat S_{\xi,\Gamma}:=\sum_{i=1}^NR_{\xi,\Gamma}^{(i)^T}(\frac{\lambda_i}{\mu_i} S_{\xi,\Gamma}^{(i)})R_{\xi,\Gamma}^{(i)}
=R_{\xi,\Gamma}^T(\frac\lambda\mu S_{\xi,\Gamma})R_{\xi,\Gamma}.\]
Consequently, the preconditioner in $\widehat W_\Gamma$ is now given by:
\begin{align}\label{eq:totalpressuresolver}
  M_{\xi_\Gamma}^{-1} :=R_{\xi,D,\Gamma}^T (\frac\lambda\mu S_{\xi,\Gamma})^{-1}R_{\xi,D,\Gamma}.
\end{align}

\subsection{The pressure sub-solver}\label{subsec:pressure}
We proceed similarly for the pressure sub-solver. We define first the local pressure Schur complement
 \[S_{p,\Gamma}^{(i)}=E_{\Delta\Delta}^{(i)}-E_{\Delta I}^{(i)}E_{II}^{(i)}E_{ I\Delta}^{(i)},\]
for $i=1,\cdots,N$, and then define
\[S_{p,\Gamma}:=\left[\begin{array}{cccc}
S_{p,\Gamma}^{(1)}&&&\\
&S_{p,\Gamma}^{(2)}&& \\
&&\ddots&\\
&&&S_{p,\Gamma}^{(N)}
\end{array}\right].\]
As for $\widehat S_{\xi,\Gamma}$, we define $\widehat S_{p,\Gamma}$ corresponding to the space $\widehat Q_\Gamma$ by
\[\widehat S_{p,\Gamma}:=\sum_{i=1}^NR_{p,\Gamma}^{(i)^T}S_{p,\Gamma}^{(i)}R_{p,\Gamma}^{(i)}=R_{p,\Gamma}^TS_{p,\Gamma}R_{p,\Gamma}.\]
Furthermore, the intermediate Schur complement $\widetilde S_{p,\Gamma}$, defined over the partially assembled space $\widetilde Q_\Gamma$, is obtained by assembling only the primal variables of $S_{p,\Gamma}^{(i)}$
\begin{equation}\label{relation:Sp}
\widetilde S_{p,\Gamma}:=\overline
R_{p,\Gamma}^TS_{p,\Gamma}\overline R_{p,\Gamma}.
\end{equation}
The Schur complements $\widehat S_{p,\Gamma}$ and $\widetilde S_{p,\Gamma}$ satisfy the relation
$  \widehat S_{p,\Gamma}=\widetilde R_{p,\Gamma}^T \widetilde S_{p,\Gamma}\widetilde R_{p,\Gamma}.$
We then define the pressure sub-solver as
\begin{equation}\label{eq:pressuresolver}
  M_{p_\Gamma}^{-1} =\widetilde R_{p,D,\Gamma}^T\widetilde S_{p,\Gamma}^{-1}\widetilde R_{p,D,\Gamma}.
\end{equation}

\subsection{The Lagrange multiplier sub-solver}\label{subsec:Lagrangemultipliter}
Here, we define the space
\begin{equation}
\widetilde{V}=\mathbf{V}_I\bigoplus W_I\bigoplus Q_I\bigoplus \mathbf{V}_\Delta\bigoplus\mathbf{V}_\Pi
\end{equation}
and its subspace
\begin{equation}
  \widetilde{V}_0=\Big\{B_{II}\mathbf v_I+B_{I\Delta}\mathbf v_\Delta+B_{I\Pi}\mathbf{v}_\Pi=0\Big\}.
\end{equation}
Define $\widetilde R_\Delta: \widetilde V \rightarrow \mathbf V_\Delta$ as the operator mapping $v=(\mathbf{v}_I, \xi_I, p_I, \mathbf{v}_\Delta,\mathbf{v}_\Pi)\in \widetilde V$ to $\widetilde R_\Delta v=\mathbf{v}_\Delta$.
Following \cite{TL2015,LT2013,TL2013}, we construct the Dirichlet preconditioner for the Lagrange multiplier by defining an extension
$H_\Delta^{(i)}:\mathbf V_\Delta^{(i)}\to\mathbf V_\Delta^{(i)}$ of
$\mathbf v_\Delta^{(i)}\in\mathbf V_\Delta^{(i)}$ as
\begin{equation}\label{eq:HDelta}
\left[
  \begin{array}{cc}
    A_{II}^{(i)} & A_{I\Delta}^{(i)} \\
    A_{\Delta I}^{(i)} & A_{\Delta\Delta}^{(i)}
  \end{array}
  \right]
  \left[\begin{array}{c}\mathbf v_I^{(i)}\\ \mathbf v_\Delta^{(i)}
  \end{array}
  \right]
  =
  \left[
  \begin{array}{c}
    0 \\
    H_\Delta^{(i)} \mathbf v_\Delta^{(i)}
  \end{array}
  \right].
\end{equation}
Multiplying $H_\Delta^{(i)}$ by the vector $ \mathbf v_\Delta^{(i)} $ requires solving a subdomain elliptic problem on $\Omega_i $ with prescribed boundary velocity $\mathbf v_\Delta^{(i)} $ and $\mathbf v_\Pi^{(i)}=\mathbf 0$.
Denoting by $H_\Delta$ the direct sum of $H_\Delta^{(i)}$, $i=1,\dots,N$, we define the Dirichlet preconditioner as
\begin{equation}\label{eq:Dirichlet}
  M_{\lambda_\Delta}^{-1}=B_{\Delta,D}H_\Delta B_{\Delta,D}^T.
\end{equation}

\subsection{Interface norms and seminorms}
In this section, we introduce some interface norms used in the analysis.
We define the following
interface norms for the pressure variables:
\[
\begin{aligned}
\|p_\Gamma^{(i)}\|_{S_{p,\Gamma}^{(i)}}^2 & :=p_\Gamma^{(i)^T}S_{p,\Gamma}^{(i)}p_\Gamma^{(i)}\quad \forall p_\Gamma^{(i)}\in Q_\Gamma^{(i)},\\
\|p_\Gamma\|_{S_{p,\Gamma}}^2 & :=p_\Gamma^T S_{p,\Gamma}p_\Gamma=\sum_{i=1}^N\|p_\Gamma^{(i)}\|_{S_{p,\Gamma}^{(i)}}^2 \quad \forall p_\Gamma\in Q_\Gamma,\\
\|p_\Gamma\|_{\widehat S_{p,\Gamma}}^2&:=p_\Gamma^TR_{p,\Gamma}^TS_{p,\Gamma} R_{p,\Gamma} p_\Gamma=\|R_{p,\Gamma}p_\Gamma\|_{S_{p,\Gamma}}^2
\quad \forall p_\Gamma\in\widehat Q_\Gamma
,\\
\|p_\Gamma\|_{\widetilde S_{p,\Gamma}}^2&:=p_\Gamma^T\overline R_{p,\Gamma}^TS_{p,\Gamma}
\overline R_{p,\Gamma}p_\Gamma=\|\overline R_{p,\Gamma}p_\Gamma\|_{S_{p,\Gamma}}^2 \quad \forall p_\Gamma\in\widetilde Q_\Gamma,\\
\|p_\Gamma^{(i)}\|_{\mathbf E_p(\Gamma_i)}&:=\inf_{\substack{q^{(i)}\in Q^{(i)}\\ q^{(i)}|_{\Gamma_i}=p_\Gamma^{(i)}}}
\Big(\|\kappa_i^{\frac12}\nabla q^{(i)}\|_{L^2(\Omega_i)}^2+\frac{2\alpha_i^2}{\lambda_i}\|q^{(i)}\|_{L^2(\Omega_i)}^2\Big)^{\frac12}\quad \forall p_\Gamma^{(i)}\in Q_\Gamma^{(i)},\\
\|p_\Gamma\|_{\mathbf E_p(\Gamma)}^2&:=\sum_{i=1}^N \|p_\Gamma^{(i)}\|_{\mathbf E_p(\Gamma_i)}^2\quad \forall p_\Gamma\in Q_\Gamma.
\end{aligned}\]
Analogously, we define the following interface norms for the total pressure variable
\[
\begin{aligned}
\|\xi_\Gamma^{(i)}\|_{S_{\xi,\Gamma}^{(i)}}^2 & :=\xi_\Gamma^{(i)^T}S_{\xi,\Gamma}^{(i)}
\xi_\Gamma^{(i)}\quad \forall \xi_\Gamma^{(i)}\in W_\Gamma^{(i)},\\
\|\xi_\Gamma\|_{S_{\xi,\Gamma}}^2 & :=\xi_\Gamma^T S_{\xi,\Gamma}\xi_\Gamma=\sum_{i=1}^N\|\xi_\Gamma^{(i)}\|_{S_{\xi,\Gamma}^{(i)}}^2
\quad \forall \xi_\Gamma\in W_\Gamma,\\
\|\xi_\Gamma\|_{\widehat S_{\xi,\Gamma}}^2 & :=\xi_\Gamma^TR_{\xi,\Gamma}^TS_{\xi,\Gamma} R_{\xi,\Gamma} \xi_\Gamma=\|R_{\xi,\Gamma}\xi_\Gamma\|_{S_{\xi,\Gamma}}^2
\quad \forall \xi_\Gamma\in\widehat W_\Gamma,\\
\|\xi_\Gamma^{(i)}\|_{\mathbf E_\xi(\Gamma_i)} & :=\inf_{\substack{\eta^{(i)}\in W^{(i)}\\ \eta^{(i)}|_{\Gamma_i}=\xi_\Gamma^{(i)}}}
\|\lambda_i^{-\frac12}\xi^{(i)}\|_{L^2(\Omega_i)}\quad \forall \xi_\Gamma^{(i)}\in W_\Gamma^{(i)},\\
\|\xi_\Gamma\|_{\mathbf E_\xi(\Gamma)}^2 & :=\sum_{i=1}^N \|\xi_\Gamma^{(i)}\|_{\mathbf E_\xi(\Gamma_i)}^2\quad \forall \xi_\Gamma\in W_\Gamma.
\end{aligned}
\]

The following Lemma establishes the equivalence of the $S_{\xi,\Gamma}$- and $\mathbf E_\xi(\Gamma)$-norms, as well as of the $S_{p,\Gamma}$- and $\mathbf E_p(\Gamma)$-norms. This result is essentially found in, e.g., Bramble and Pasciak \cite{BP1990}.
\begin{lemma}\label{lem:eqnorms}
There exist two positive constants $C_1$ and $C_0$ such that for any $\xi_\Gamma\in W_\Gamma$ and any $p_\Gamma\in Q_\Gamma$, there hold the inequalities
\begin{align}
 C_1 \|\xi_\Gamma\|_{S_{\xi,\Gamma}}\leq \|\xi_\Gamma\|_{\mathbf E_\xi(\Gamma)}\leq C_0\|\xi_\Gamma\|_{S_{\xi,\Gamma}},\label{ineq:eqnorms-xi}\\
 C_1 \|p_\Gamma\|_{S_{p,\Gamma}}\leq \|p_\Gamma\|_{\mathbf E_p(\Gamma)}\leq C_0\|p_\Gamma\|_{S_{p,\Gamma}}.\label{ineq:eqnorms-p}
\end{align}
\end{lemma}

In order to prove an upper bound in our main result, we introduce the averaging operators $E_{\xi,D}=R_{\xi,\Gamma}R_{\xi,D,\Gamma}^T$, \  $E_{p,D}=\widetilde R_{p,\Gamma}\widetilde R_{p,D,\Gamma}^T$, and
\[E_{\xi,p,D}=\left[\begin{array}{cc}E_{\xi,D}&\\&E_{p,D}\end{array}\right].\]
The averaging operators $E_{\xi,D}$ and $E_{p,D}$ satisfy the following stability inequalities.
\begin{lemma}\label{lem:avgsbl}
The following two inequalities hold:
\begin{align}
\|E_{\xi,D}\eta_\Gamma\|_{\mathbf E_\xi(\Gamma)}&\leq C_0\|\eta_\Gamma\|_{\mathbf E_\xi(\Gamma)}\quad \forall \eta_\Gamma\in W_\Gamma,\label{ineq:Exi}\\
  \|\overline R_{p,\Gamma}(E_{p,D}q_\Gamma)\|_{\mathbf E_p(\Gamma)}&\leq C_0\left(1+\log\frac Hh\right)\|\overline R_{p,\Gamma}q_\Gamma\|_{\mathbf E_p(\Gamma)}\quad \forall q_\Gamma\in \widetilde Q_\Gamma. \label{ineq:Ep}
\end{align}
\end{lemma}
\begin{proof}
\eqref{ineq:Exi} can be easily verified. \eqref{ineq:Ep} is a well-known result, see e.g \cite{TL2015}.
\end{proof}

\section{Convergence rate estimates}
In this section, we shall proceed with estimating the condition number of the preconditioned system
$M^{-1}G$.

\subsection{Conditioner number estimates for the block BDDC subsystem}

We consider first the preconditioner for the total pressure and pressure block
\begin{equation}\label{eq:blockBDDC}
  M_{\xi,p,\Gamma}^{-1}:=\left[\begin{array}{cc}
                   M_{\xi_\Gamma}^{-1}& \\
                   & M_{p_\Gamma}^{-1}
                 \end{array}\right].
\end{equation}
We define the block Schur complements
\begin{equation}\label{eq:coupleSc}
\widehat S_{\xi,p,\Gamma}=\left[\begin{array}{cc}
\frac\lambda\mu\widehat S_{\xi,\Gamma}& \\
&  \widehat S_{p,\Gamma}
\end{array}
\right],\qquad
\widetilde S_{\xi,p,\Gamma}=\left[\begin{array}{cc}
\frac\lambda\mu S_{\xi,\Gamma}& \\
&  \widetilde S_{p,\Gamma}
\end{array}
\right].\end{equation}

\begin{lemma}\label{lem:BDDClowerbound}
For any $\bm p_\Gamma=(\xi_\Gamma,p_\Gamma)\in\widehat W_\Gamma\times\widehat Q_\Gamma$, we have
\begin{equation}\label{ineq:BDDClowerbound}
  \langle \bm p_\Gamma,\bm p_\Gamma\rangle_{\widehat S_{\xi, p,\Gamma}}\leq \langle \bm p_\Gamma, M_{\xi,p,\Gamma}^{-1}\widehat S_{\xi,p,\Gamma}\bm p_\Gamma \rangle_{\widehat S_{\xi,p,\Gamma}}.
\end{equation}
\end{lemma}
\begin{proof}
Given $\xi_\Gamma\in\widehat W_\Gamma$ and $p_\Gamma\in Q_\Gamma$, we define $\zeta_\Gamma= (\frac\lambda\mu S_{\xi,\Gamma})^{-1} R_{\xi,D,\Gamma}
(\frac\lambda\mu\widehat S_{\xi,\Gamma})\xi_\Gamma\in W_\Gamma$ and
$w_\Gamma=\widetilde S_{p,\Gamma}^{-1}\widetilde R_{p,D,\Gamma}\widehat S_{p,\Gamma}p_\Gamma\in\widetilde Q_\Gamma$.
Then, from the relations \[R_{\xi,\Gamma}^TR_{\xi,D,\Gamma}= R_{\xi,D,\Gamma}^TR_{\xi,\Gamma}=\mathcal I \ \text{    and    }\
\widetilde R_{p,\Gamma}^T\widetilde R_{p,D,\Gamma}=\widetilde R_{p,D,\Gamma}^T\widetilde R_{p,\Gamma}=\mathcal I,\]
with
$\mathcal I$ the appropriate identity matrix, we derive that
\begin{equation}\label{ineq:BDDClowerbound1}
\begin{split}
  & \langle \bm p_\Gamma,\bm p_\Gamma\rangle_{\widehat S_{\xi,p,\Gamma}}\\
  =\,\,&\langle \xi_\Gamma,\xi_\Gamma\rangle_{\frac\lambda\mu\widehat S_{\xi,\Gamma}}+ \langle p_\Gamma,p_\Gamma\rangle_{\widehat S_{p,\Gamma}}\\
  =\,\,&\xi_\Gamma^T(\frac\lambda\mu\widehat S_{\xi,\Gamma})R_{\xi,D,\Gamma}^T R_{\xi,\Gamma}\xi_\Gamma+p_\Gamma^T\widehat S_{p,\Gamma}\widetilde R_{p,D,\Gamma}^T\widetilde R_{p,\Gamma}p_\Gamma \\
  =\,\,&\xi_\Gamma^T(\frac\lambda\mu\widehat S_{\xi,\Gamma})R_{\xi,D,\Gamma}^T(\frac\lambda\mu S_{\xi,\Gamma})^{-1}
  (\frac\lambda\mu S_{\xi,\Gamma}) R_{\xi,\Gamma}\xi_\Gamma
  +p_\Gamma^T\widehat S_{p,\Gamma}\widetilde R_{p,D,\Gamma}^T\widetilde S_{p,\Gamma}^{-1}\widetilde S_{p,\Gamma}\widetilde R_{p,\Gamma}p_\Gamma \\
  =\,\,&\langle \zeta_\Gamma, R_{\xi,\Gamma}\xi_\Gamma\rangle_{\frac\lambda\mu S_{\xi,\Gamma}}+\langle w_\Gamma, \widetilde R_{p,\Gamma}
  p_\Gamma\rangle_{\widetilde S_{p,\Gamma}}.
  \end{split}
\end{equation}
Utilizing the Cauchy-Schwarz inequality and the relation $\frac\lambda\mu\widehat S_{\xi,\Gamma}=R_{\xi,\Gamma}^T
(\frac\lambda\mu S_{\xi,\Gamma})R_{\xi,\Gamma}$, we arrive at
\begin{equation}\label{ineq:BDDClowerbound2}
\langle \zeta_\Gamma, R_{\xi,\Gamma}\xi_\Gamma\rangle_{\frac\lambda\mu S_{\xi,\Gamma}}^2\leq
\langle \zeta_\Gamma, \zeta_\Gamma\rangle_{\frac\lambda\mu S_{\xi,\Gamma}}\langle
R_{\xi,\Gamma}\xi_\Gamma,R_{\xi,\Gamma}\xi_\Gamma\rangle_{\frac\lambda\mu S_{\xi,\Gamma}}
=\langle \zeta_\Gamma, \zeta_\Gamma\rangle_{\frac\lambda\mu S_{\xi,\Gamma}}\langle \xi_\Gamma, \xi_\Gamma\rangle_{\frac\lambda\mu \widehat S_{\xi,\Gamma}}.
\end{equation}
Similarly, for pressure variable, we have
 \begin{equation}\label{ineq:BDDClowerbound3}
\langle w_\Gamma, \widetilde R_{p,\Gamma} p_\Gamma\rangle_{\widetilde S_{p,\Gamma}}^2\leq\langle w_\Gamma,w_\Gamma\rangle_{\widetilde S_{p,\Gamma}}
\langle p_\Gamma, p_\Gamma\rangle_{\widehat S_{p,\Gamma}}.
\end{equation}
Consequently, from \eqref{ineq:BDDClowerbound1}, \eqref{ineq:BDDClowerbound2}, \eqref{ineq:BDDClowerbound3} and Young's inequality, we obtain
 \begin{equation}\label{ineq:BDDClowerbound4}
\langle \bm p_\Gamma,\bm p_\Gamma\rangle_{\widehat S_{\xi,p,\Gamma}}\leq
2\Big(\langle \zeta_\Gamma, \zeta_\Gamma\rangle_{\frac\lambda\mu S_{\xi,\Gamma}}+\langle w_\Gamma,w_\Gamma\rangle_{\widetilde S_{p,\Gamma}}\Big).
\end{equation}
Finally, using \eqref{eq:totalpressuresolver}\,\eqref{eq:pressuresolver}\,\eqref{eq:blockBDDC}, we have
\begin{equation}\label{ineq:BDDClowerbound5}
 \begin{split}
 &\langle \zeta_\Gamma, \zeta_\Gamma\rangle_{\frac\lambda\mu S_{\xi,\Gamma}}+\langle w_\Gamma,w_\Gamma\rangle_{\widetilde S_{p,\Gamma}}\\
=\,\,&\xi_\Gamma^T(\frac\lambda\mu \widehat S_{\xi,\Gamma})R_{\xi,D,\Gamma}^T
(\frac\lambda\mu S_{\xi,\Gamma})^{-1} R_{\xi,D,\Gamma}(\frac\lambda\mu\widehat S_{\xi,\Gamma})\xi_\Gamma
+p_\Gamma^T\widehat S_{p,\Gamma}\widetilde R_{p,D,\Gamma}^T\widetilde S_{p,\Gamma}^{-1}\widetilde R_{p,D,\Gamma}\widehat S_{p,\Gamma}p_\Gamma \\
=\,\,&\langle\xi_\Gamma,M_{\xi_\Gamma}^{-1}(\frac\lambda\mu\widehat S_{\xi,\Gamma})\xi_\Gamma
\rangle_{\frac\lambda\mu\widehat S_{\xi,\Gamma}}+\langle p_\Gamma,M_{p_\Gamma}^{-1}\widehat S_{p,\Gamma}p_\Gamma\rangle_{\widehat S_{p,\Gamma}}
=\langle \bm p_\Gamma, M_{\xi,p,\Gamma}^{-1}\widehat S_{\xi,p,\Gamma}\bm p_\Gamma \rangle_{\widehat S_{\xi,p,\Gamma}}.
\end{split}
\end{equation}
Substituting  \eqref{ineq:BDDClowerbound5} into \eqref{ineq:BDDClowerbound4} gives the desired result, and the proof is completed.
\end{proof}

We are now in a 
position to give an upper bound estimate for the eigenvalues of $M_{\xi,p,\Gamma}^{-1}\widehat S_{\xi,p,\Gamma}$. To this end, we first establish the following Lemma.
\begin{lemma}\label{lem:BDDCupperboundlem}
For any $\bm p_\Gamma=(\xi_\Gamma,p_\Gamma)\in\widehat W_\Gamma\times\widehat Q_\Gamma$, 
\begin{equation}\label{ineq:BDDClowerboundlem}
  \|E_{\bm p,D}\bm p_\Gamma\|_{\widetilde S_{\xi, p,\Gamma}}\leq C_1\left(1+\log\frac Hh\right)\|\bm p_\Gamma\|_{\widetilde S_{\xi,p,\Gamma}}.
\end{equation}
\end{lemma}
\begin{proof}
From the definition of the averaging operator $E_{\bm p,D}$ and Lemma \ref{lem:eqnorms}, one has
\begin{equation}\label{ineq:BDDClowerboundlem1}
 \begin{split}
  \|E_{\xi,p,D}\bm p_\Gamma\|_{\widetilde S_{\xi, p,\Gamma}}^2
  &=\|E_{\xi,D}\xi_\Gamma\|_{\frac\lambda\mu S_{\xi,\Gamma}}^2+\|E_{p,D}p_\Gamma\|_{\widetilde S_{p,\Gamma}}^2 \\
  &=\|E_{\xi,D}\xi_\Gamma\|_{\frac\lambda\mu S_{\xi,\Gamma}}^2+\|\overline R_{p,\Gamma}(E_{p,D}p_\Gamma)\|_{S_{p,\Gamma}}^2 \\
  &\leq C_1\Big(\|E_{\xi,D}\xi_\Gamma\|_{\mathbf E_\xi(\Gamma)}^2+\|\overline R_{p,\Gamma}(E_{p,D}p_\Gamma)\|_{\mathbf E_p(\Gamma)}^2\Big).
  \end{split}
\end{equation}
On the other hand, we use \eqref{ineq:Exi} to derive
\begin{equation}\label{ineq:BDDClowerboundlem2}
  \|E_{\xi,D}\xi_\Gamma\|_{\mathbf E_\xi(\Gamma)}^2\leq C_1\|\xi_\Gamma\|_{\mathbf E_\xi(\Gamma)}^2
  \leq c\|\xi_\Gamma\|_{\frac\lambda\mu S_{\xi,\Gamma}}^2.
\end{equation}
Likewise, for the pressure variable, we can deduce from \eqref{ineq:Ep}\eqref{ineq:eqnorms-p} and \eqref{relation:Sp} that
\begin{equation}\label{ineq:BDDClowerboundlem3}
 \|\overline R_{p,\Gamma}(E_{D,p}p_\Gamma)\|_{\mathbf E_p(\Gamma)}^2\leq C_1\left(1+\log\frac Hh\right)^2\|p_\Gamma\|_{\widetilde S_{p,\Gamma}}^2.
\end{equation}
Hence, we obtain from \eqref{ineq:BDDClowerboundlem1}-\eqref{ineq:BDDClowerboundlem3} that
\begin{equation}\label{ineq:BDDClowerboundlem4}
\begin{split}
\|E_{\bm p,D}\bm p_\Gamma\|_{\widetilde S_{\xi, p,\Gamma}}^2
&\leq C_1\left(1+\log\frac Hh\right)^2\left(\|\xi_\Gamma\|_{\frac\lambda\mu S_{\xi,\Gamma}}^2+\|p_\Gamma\|_{\widetilde S_{p,\Gamma}}^2\right) \\
&=C_1\left(1+\log\frac Hh\right)^2\|\bm p_\Gamma\|_{\widetilde S_{\xi, p,\Gamma}}^2.
\end{split}
\end{equation}
The proof is complete.
\end{proof}

\begin{lemma}\label{lem:BDDCupperbound}
For any $\bm p_\Gamma=(\xi_\Gamma,p_\Gamma)\in\widehat W_\Gamma\times\widehat Q_\Gamma$, we have
\begin{equation}\label{ineq:BDDCupperbound}
  \langle\bm p_\Gamma,M_{\xi,p,\Gamma}^{-1}\widehat S_{\xi,p,\Gamma}\bm p_\Gamma\rangle_{\widehat S_{\xi,p,\Gamma}}\leq C_1\left(1+\log\frac Hh\right)^2\langle\bm p_\Gamma,\bm p_\Gamma\rangle_{\widehat S_{\xi,p,\Gamma}}.
\end{equation}
\end{lemma}
\begin{proof}
Given $\xi_\Gamma\in\widehat W_\Gamma$ and $p_\Gamma\in\widehat Q_\Gamma$, we define
$\zeta_\Gamma= (\frac\lambda\mu S_{\xi,\Gamma})^{-1} R_{\xi,D,\Gamma}
(\frac\lambda\mu\widehat S_{\xi,\Gamma})\xi_\Gamma\in W_\Gamma$,
$w_\Gamma=\widetilde S_{p,\Gamma}^{-1}\widetilde R_{p,D,\Gamma}\widehat S_{p,\Gamma}p_\Gamma\in\widetilde Q_\Gamma$, and denote $\bm w_\Gamma=(\zeta_\Gamma,w_\Gamma)$. Then there holds
$\widetilde R_{\xi, p,D,\Gamma}^T\bm w_\Gamma=M_{\xi,p,\Gamma}^{-1}\widehat S_{\xi,p,\Gamma}\bm p_\Gamma$, where
\[\widetilde R_{\xi,p,D,\Gamma}:=\left[\begin{array}{cc}
R_{\xi,D,\Gamma}  & \\
& \widetilde R_{p,D,\Gamma}
\end{array}
\right],\ \text{ and we define }\ \widetilde R_{\xi,p,\Gamma}:=\left[\begin{array}{cc}
R_{\xi,\Gamma}  & \\
& \widetilde R_{p,\Gamma}
\end{array}
\right].\]
Noticing that $\widehat S_{\xi,p,\Gamma}=\widetilde R_{\xi,p,\Gamma}^T\widetilde S_{\xi,p,\Gamma}\widetilde R_{\xi,p,\Gamma}$, we derive from Lemma \ref{lem:BDDCupperboundlem} that
\begin{equation}\label{ineq:BDDCupperbound1}
\begin{split}
 &\langle M_{\xi,p,\Gamma}^{-1}\widehat S_{\xi,p,\Gamma}\bm p_\Gamma,M_{\xi,p,\Gamma}^{-1}\widehat S_{\xi,p,\Gamma}\bm p_\Gamma\rangle_{\widehat S_{\xi,p,\Gamma}}
=\,\langle \widetilde R_{\xi,p,D,\Gamma}^T\bm w_\Gamma,\widetilde R_{\xi,p,D,\Gamma}^T\bm w_\Gamma\rangle_{\widehat S_{\xi,p,\Gamma}} \\
=&\, \langle \widetilde R_{\xi,p,\Gamma}\widetilde R_{\xi,p,D,\Gamma}^T\bm w_\Gamma,\widetilde R_{\xi,p,\Gamma}\widetilde R_{\xi,p,D,\Gamma}^T\bm w_\Gamma\rangle_{\widetilde S_{\xi,p,\Gamma}}\\
=&\,\|E_{\bm p,D}\bm w_\Gamma\|_{\widetilde S_{\xi, p,\Gamma}}^2\leq C_1\left(1+\log\frac Hh\right)^2\|\bm w_\Gamma\|_{\widetilde S_{\xi, p,\Gamma}}^2.
  \end{split}
\end{equation}
We deduce from \eqref{ineq:BDDClowerbound5} that
\begin{equation}\label{ineq:BDDCupperbound2}
\|\bm w_\Gamma\|_{\widetilde S_{\xi,p,\Gamma}}^2=\langle\bm p_\Gamma,M_{\xi,p,\Gamma}^{-1}\widehat S_{\xi,p,\Gamma}\bm p_\Gamma\rangle_{\widehat S_{\xi,p,\Gamma}}.
\end{equation}
Therefore, from \eqref{ineq:BDDCupperbound1}, \eqref{ineq:BDDCupperbound2} and the Cauchy-Schwarz inequality, we obtain
 \begin{equation}\label{ineq:BDDCupperbound3}
\begin{split}
  \langle\bm p_\Gamma,M_{\xi,p,\Gamma}^{-1}\widehat S_{\xi,p,\Gamma}\bm p_\Gamma\rangle_{\widehat S_{\xi,p,\Gamma}}&\leq \langle\bm p_\Gamma,\bm p_\Gamma\rangle_{\widehat S_{\xi,p,\Gamma}}^{1/2}
  \langle M_{\xi,p,\Gamma}^{-1}\widehat S_{\xi,p,\Gamma}\bm p_\Gamma,M_{\xi,p,\Gamma}^{-1}\widehat S_{\xi,p,\Gamma}\bm p_\Gamma\rangle_{\widehat S_{\xi,p,\Gamma}}^{1/2}\\
  &\leq C_0\left(1+\log\frac Hh\right)\langle\bm p_\Gamma,\bm p_\Gamma\rangle_{\widehat S_{\xi,p,\Gamma}}^{1/2}\langle\bm p_\Gamma,M_{\xi,p,\Gamma}^{-1}\widehat S_{\xi,p,\Gamma}\bm p_\Gamma\rangle_{\widehat S_{\xi,p,\Gamma}}^{1/2},
  \end{split}
\end{equation}
and the desired result follows.
\end{proof}
Combining Lemma \ref{lem:BDDClowerbound} and Lemma \ref{lem:BDDCupperbound} yields the following Theorem.
\begin{theorem}\label{lem:BDDCConditionNumber}
  For any $\bm p_\Gamma=(\xi_\Gamma,p_\Gamma)\in\widehat W_\Gamma\times\widehat Q_\Gamma$, it holds that
  \begin{equation}\label{ineq:BDDCConditionNumber}
    C_0\langle\bm p_\Gamma,\bm p_\Gamma\rangle_{\widehat S_{\xi,p,\Gamma}}\leq\langle\bm p_\Gamma, M_{\xi,p,\Gamma}^{-1}\widehat S_{\xi, p,\Gamma}\bm p_\Gamma\rangle_{\widehat S_{\xi, p,\Gamma}}\leq C_1\left(1+\log\frac Hh\right)^2
    \langle\bm p_\Gamma,\bm p_\Gamma\rangle_{\widehat S_{\xi, p,\Gamma}}.
  \end{equation}
\end{theorem}

\subsection{Condition number stimates for FETI-DP}
In this section, we will derive a condition number bound for $M^{-1}G$.
To do this, we need to give upper and lower bounds for $\bm x^TG\bm x$, with $\bm x=\left[\begin{array}{ccc}\eta_\Gamma^T&p_\Gamma^T&\lambda_\Delta^T\end{array}\right]^T$.
As in \cite{WZSP2021}, we define the space $\widehat V_0$, which is spanned by the solutions of the equations:
\[\left[
    \begin{array}{cccccccc}
A_{II}&B_{II}^T&0&A_{I\Delta}&A_{I\Pi}& B_{\Gamma I}^T&0 & \\
B_{II}&-C_{II}&D_{II}^T&B_{I\Delta}&B_{I\Pi}&-C_{I\Gamma}&D_{\Gamma I}^T& \\
0&D_{II}&-E_{II}&0 &0 &D_{I\Gamma}&-E_{I\Gamma}& \\
A_{\Delta I}&B_{I\Delta}^T&0&A_{\Delta\Delta}&A_{\Delta\Pi}&B_{\Gamma\Delta}^T&0&B_{\Delta}^T\\
A_{\Pi I}&B_{I\Pi}&0&A_{\Pi\Delta}&A_{\Pi\Pi}&B_{\Gamma\Pi}&0&
  \end{array}
    \right]
        \left[
    \begin{array}{c}
\mathbf{u}_I\\
\xi_I\\p_I\\
\mathbf{u}_{\Delta}\\ \mathbf{u}_{\Pi}\\
\xi_{\Gamma}\\ p_{\Gamma}\\
\lambda_{\Delta}\\
      \end{array}
    \right]
    =\bm 0.\]

Define the matrices
\begin{equation}
   \widetilde{B}= \left[
    \begin{array}{ccc}
B_{II}&B_{I\Delta}&B_{I\Pi}\\
0&0 &0 \\
B_{\Gamma I}&B_{\Gamma\Delta}&B_{\Gamma\Pi}\\
0&0&0
  \end{array}
    \right],\quad
   \widetilde{B}^{(i)}= \left[
    \begin{array}{ccc}
B_{II}^{(i)}&B_{I\Delta}^{(i)}&B_{I\Pi}^{(i)}\\
0&0 &0 \\
B_{\Gamma I}^{(i)}&B_{\Gamma\Delta}^{(i)}&B_{\Gamma\Pi}^{(i)}\\
0&0&0
  \end{array}
    \right].
    \end{equation}

The following two Lemmas are similar to those proven in \cite{WZSP2021}.

\begin{lemma}\label{lem:upperbound}
For any $\bm x$ and $\left[\begin{array}{cccc}\mathbf v^T&\eta_\Gamma^T&p_\Gamma^T&\lambda_\Delta^T\end{array}\right]$ satisfy \eqref{eq:x}, it holds that
\begin{equation}\label{ineq:upperbound}
  \bm x^TGM^{-1}G\bm x\leq C_0\left(1+\log\frac Hh\right)^2\Big(\mathbf u^T A\mathbf u+\max_i\{\frac{\mu_i}{\lambda_i}\} \eta^TC\eta+p^TEp\Big).
\end{equation}
\end{lemma}
\begin{proof}
From the definition of $\widetilde V_0$, we know that for any $\bm x=\left[\begin{array}{ccc}\xi_\Gamma^T&p_\Gamma^T&\lambda_\Delta^T\end{array}\right]^T\in \widehat W_\Gamma\times\widehat Q_\Gamma\times\Lambda$, there exists
a vector $\left[\begin{array}{cccc}\mathbf v^T&\xi_\Gamma^T&p_\Gamma^T&\lambda_\Delta^T\end{array}\right]^T\in \widehat V_0$ satisfies
\begin{equation}\label{eq:x}
\widetilde B_C\mathbf v+\widetilde C\left[\begin{array}{ccc}\xi_\Gamma^T& p_\Gamma^T& \lambda_\Delta^T\end{array}\right]^T=G\bm x.
\end{equation}
Then the left-hand side of \eqref{ineq:upperbound} can be rewritten as
\begin{equation}\label{ineq:upperbound1}
 \left( \widetilde B_C\mathbf v+\widetilde C\left[\begin{array}{c}\xi_\Gamma\\ p_\Gamma\\ \lambda_\Delta\end{array}\right]\right)^T
 M^{-1}\left(\widetilde B_C\mathbf v+\widetilde C\left[\begin{array}{c}\xi_\Gamma\\ p_\Gamma\\ \lambda_\Delta\end{array}\right]\right).
\end{equation}
Given $\left[\begin{array}{ccccccc}\mathbf u_I^T& \xi_I^T& p_I^T& u_\Delta^T& u_\Pi^T& \xi_\Gamma^T& p_\Gamma^T\end{array}\right]^T$, we divide $\widetilde B_C\mathbf v+\widetilde C\left[\begin{array}{ccc}\xi_\Gamma^T& p_\Gamma^T& \lambda_\Delta^T\end{array}\right]^T$ into two components for consideration.
To alleviate the notation, we denote $g_{\bm p_\Gamma}$ as
\[B_{\Gamma I}\mathbf u_I-\left[
      \begin{array}{cc}
        C_{\Gamma I} & -D_{\Gamma I}^T \\
        -D_{\Gamma I} & E_{\Gamma I}
      \end{array}\right]\left[\begin{array}{c}\eta_I\\ p_I\end{array}\right]
      +B_{\Gamma \Delta}\mathbf u_\Delta+B_{\Gamma \Pi}\mathbf u_\Pi-
      \left[
      \begin{array}{cc}
        C_{\Gamma \Gamma} & -D_{\Gamma \Gamma}^T \\
        -D_{\Gamma \Gamma} & E_{\Gamma \Gamma}
      \end{array}\right]\left[\begin{array}{c}\eta_\Gamma\\ p_\Gamma\end{array}\right].
      \]

Since $\left[\begin{array}{cccc}\mathbf v^T&\xi_\Gamma^T&p_\Gamma^T&\lambda_\Delta^T\end{array}\right]^T$ belongs to $\widehat V_0$, there holds the identity
\begin{equation}\label{eq:g=0}
B_{II}\mathbf u_I-\left[
      \begin{array}{cc}
        C_{II} & -D_{II}^T \\
        -D_{II} & E_{II}
      \end{array}\right]\left[\begin{array}{c}\xi_I\\ p_I\end{array}\right]
      +B_{I\Delta}\mathbf u_\Delta+B_{I\Pi}\mathbf u_\Pi-
      \left[
      \begin{array}{cc}
        C_{I\Gamma} & -D_{I\Gamma}^T \\
        -D_{I\Gamma} & E_{I\Gamma}
      \end{array}\right]\left[\begin{array}{c}\xi_\Gamma\\ p_\Gamma\end{array}\right]=\bm 0.
      \end{equation}

The preconditioner for this first component is defined as $M_{\xi,p,\Gamma}^{-1}$, which is given in \eqref{eq:blockBDDC}.
Using the block Cholesky factorization, we have
\begin{equation}\label{eq:g=1}
\left[\begin{array}{cc} C_{II} & C_{I\Gamma} \\ C_{\Gamma I} & C_{\Gamma\Gamma}\end{array}\right]
= \left[\begin{array}{cc} \mathcal I &  \\ C_{\Gamma I}C_{II}^{-1} & \mathcal I\end{array}\right]
\left[\begin{array}{cc} C_{II} &  \\  & \widehat S_{\xi,\Gamma}\end{array}\right]
\left[\begin{array}{cc} \mathcal I & C_{\Gamma I}C_{II}^{-1} \\  & \mathcal I\end{array}\right],
\end{equation}
\begin{equation}\label{eq:g=2} 
\left[\begin{array}{cc} E_{II} & E_{I\Gamma} \\ E_{\Gamma I} & E_{\Gamma\Gamma}\end{array}\right]
= \left[\begin{array}{cc} \mathcal I &  \\ E_{\Gamma I}E_{II}^{-1} & \mathcal I\end{array}\right]
\left[\begin{array}{cc} E_{II} &  \\  & \widehat S_{p,\Gamma}\end{array}\right]
\left[\begin{array}{cc} \mathcal I & E_{\Gamma I}E_{II}^{-1} \\  & \mathcal I\end{array}\right],
\end{equation}
where $\mathcal I$ is the appropriate identity matrix.

Consequently, we deduce from Theorem \ref{lem:BDDCConditionNumber}, identities \eqref{eq:g=0} - \eqref{eq:g=2}, and Young's inequality that
\[
   g_{\bm p_\Gamma}^T  M_{\xi,p,\Gamma}^{-1}g_{\bm p_\Gamma}
   \leq\, C_0\left(1+\log\frac Hh\right)^2g_{\bm p_\Gamma}^T  \left[\begin{array}{cc}
           (\frac\lambda\mu\widehat S_{\xi,\Gamma})^{-1} &  \\
            & \widehat S_{\xi,\Gamma}^{-1}
         \end{array}\right]
   g_{\bm p_\Gamma} =
\]
\[
\!C_0\!\left(\!1+\log\frac Hh\right)^2\!\!\left(\!\widetilde B\mathbf u+ \left[
      \begin{array}{cc}
        \!\!C &\!\!-D^T\!\! \\
        \!\!-D\!\! &\!\!E\!\!
      \end{array}\right]\left[\begin{array}{c}\!\!\xi\!\!\\ \!\!p\!\!\end{array}\right]\right)^T\!\!
    \left[
      \begin{array}{cc}
        \!\!\frac\lambda\mu C\!\! &   \\
          & \!\!E\!\!
      \end{array}\right]^{-1}
      \!\!\left(\!\widetilde B\mathbf u+ \left[
      \begin{array}{cc}
       \!\! C \!\!&\!\! -D^T\!\! \\
        \!\!-D\!\! & \!\!E\!\!
      \end{array}\right]\left[\begin{array}{c}\!\!\xi\!\!\\ \!\!p\!\!\end{array}\right]\right)
      \]
\begin{equation}
\begin{split}
      \leq\,& 2C_0\left(1+\log\frac Hh\right)^2\Bigg((\widetilde B\mathbf u)^T\left[
      \begin{array}{cc}
        \frac\lambda\mu C &   \\
          & E
      \end{array}\right]^{-1}\widetilde B\mathbf u \\
      &+\left[\begin{array}{cc}
               \xi^T & p^T
             \end{array}\right]\left[
      \begin{array}{cc}
        C & -D^T \\
        -D & E
      \end{array}\right]
      \left[
      \begin{array}{cc}
       \frac\lambda\mu C &   \\
          & E
      \end{array}\right]^{-1}
      \left[
      \begin{array}{cc}
        C & -D^T \\
        -D & E
      \end{array}\right]\left[\begin{array}{c}\xi\\ p\end{array}\right]\Bigg)\\
      :=&\,2C_0\left(1+\log\frac Hh\right)^2(\mathscr J_1+\mathscr J_2).
\end{split}\label{ineq:upperbound3}
\end{equation}

We next estimate $\mathscr J_1$ and $\mathscr J_2$.  For the first term $\mathscr J_1$, one has
\begin{equation}\label{ineq:upperbound4}
 \mathscr J_1\leq \sup_{\bm q=(\eta,q)\in W\times Q}\frac{((\widetilde B\mathbf u)^T\bm q)^2}{\eta^T\frac\lambda\mu C\eta}=\sup_{\bm q=(\eta,q)\in W\times Q}
  \frac{(\sum_i(\widetilde B^{(i)}\mathbf u_i)^T\bm q_i)^2}{\sum_i\eta_i^T\frac{\lambda_i}{\mu_i} C^{(i)}\eta_i}.
\end{equation}
By the definition of $\widetilde B^{(i)}$ and the Cauchy-Schwarz inequality, we claim that
\begin{equation}\label{ineq:upperbound5}
  \left|(\widetilde B^{(i)}\mathbf u_i)^T\bm q_i\right|=\left|-\int_{\Omega_i}\dvg \mathbf u_i\eta_i dx\right|\leq c\Big(\mathbf u_i^T
  A^{(i)}\mathbf u_i\Big)^{\frac12}
  \Big(\eta_i^T\frac{\lambda_i}{\mu_i} C^{(i)}\eta_i\Big)^{\frac12}.
\end{equation}
Substituting \eqref{ineq:upperbound5} into \eqref{ineq:upperbound4} yields
\begin{equation}\label{ineq:upperbound6}
  \mathscr J_1\leq c\mathbf u^T A\mathbf u.
\end{equation}
We now turn to the term $\mathscr J_2$. For any $(\eta,q)\in W\times Q$, by Lemma \ref{lem:SPD2pressures}, the Cauchy-Schwarz inequality
and Lemma \ref{lem:bound2pressures}, we obtain
\begin{equation}\label{ineq:upperbound+1}
 \left( \left[\begin{array}{cc}
               \xi^T & p^T
             \end{array}\right]\left[
      \begin{array}{cc}
        C & -D^T \\
        -D & E
      \end{array}\right]\left[\begin{array}{c}\eta\\ q\end{array}\right]\right)^2
      \leq c\left(\xi^T\frac\lambda\mu C\xi+p^TEp\right)\left(\eta^T\frac\lambda\mu C\eta+q^TEq\right).
\end{equation}
Therefore, we derive from \eqref{ineq:upperbound+1} that
\begin{equation}\label{ineq:upperbound7}
  \begin{split}
    \mathscr J_2&\leq c\sup_{(\eta,q)\in W\times Q}\frac{(\xi^T\frac\mu\lambda C\xi+p^TEp)(\eta^T\frac\lambda\mu C\eta+q^TEq)}
    {\eta^T\frac\lambda\mu C\eta+q^TEq}\\
    &= c\sup_{(\eta,q)\in W\times Q}\frac{\left[\sum_i(\frac{\mu_i}{\lambda_i}\xi_i^TC^{(i)}\xi_i
    +p_i^TE^{(i)}p_i)\right]\cdot\left[\sum_i(\frac{\lambda_i}{\mu_i}\eta_i^TC^{(i)}\eta_i+q_i^TE^{(i)}q_i)\right]}
    {\sum_i(\eta_i^T\frac{\lambda_i}{\mu_i} C^{(i)}\eta_i+q_i^TE^{(i)}q_i)}\\
    &\leq c\Big(\max_i\{\frac{\mu_i}{\lambda_i}\} \eta^TC\eta+p^TEp\Big).
  \end{split}
\end{equation}
From \eqref{ineq:upperbound3}, \eqref{ineq:upperbound6} and \eqref{ineq:upperbound7}, we obtain an upper bound for the preconditioned block associated with pressure and total-pressure variables, namely,
\begin{equation}\label{ineq:upperbound8}
  g_{\bm p_\Gamma}^T M_{\xi,p,\Gamma}^{-1}g_{\bm p_\Gamma}\leq C_0\left(1+\log\frac Hh\right)^2
  \Big(\mathbf u^TA\mathbf u+\max_i\{\frac{\mu_i}{\lambda_i}\} \eta^TC\eta+p^TEp\Big).
\end{equation}
When the primal variables are appropriately selected, the bound for the Lagrange multiplier sub-solver can be found in \cite{KW2006,TL2015}.
The proof is completed.
\end{proof}

\begin{lemma}\label{lem:lowerbound}
For any $\bm x$ and $\left[\begin{array}{cccc}\mathbf v^T&\eta_\Gamma^T&p_\Gamma^T&\lambda_\Delta^T\end{array}\right]^T$ satisfy \eqref{eq:x}, we have
\begin{equation}\label{ineq:lowerbound}
 C_1(\mathbf u^TA\mathbf u+\eta^TC\eta+p^TEp) \leq\bm x^TGM^{-1}G\bm x.
\end{equation}
\end{lemma}
\begin{proof}
Denote $\bm y=G\bm x=\left[\begin{array}{cc}g_{\bm p_\Gamma}^T& g_\lambda^T\end{array}\right]^T$ and set $\mathbf u_\Delta^{(1)}=B_{\Delta,D}^Tg_\lambda$, $\ \mathbf u_\Pi^{(1)}=\mathbf 0$,  $\ \xi^{(1)}=0$, $\ p^{(1)}=0,\ \lambda_\Delta^{(1)}=0$,
and set $\mathbf u_I^{(1)}$ to be the solution of equations \eqref{eq:HDelta} on each subdomains.  The corresponding global variables are
denoted by $\mathbf u^{(1)}$, $\xi^{(1)}$ and $p^{(1)}$, respectively. Using the fact $B_\Delta\mathbf u_\Delta^{(1)}=B_\Delta B_{\Delta,D}^Tg_\lambda=g_\lambda$, we arrive at
\begin{equation}\label{ineq:lowerbound1}
\widetilde B_C\mathbf v^{(1)}=
\left[\begin{array}{c}B_{\Gamma I}\mathbf u_I^{(1)}+B_{\Gamma\Delta}\mathbf u_\Delta^{(1)}\\ 0\\ g_\lambda\end{array}\right].
\end{equation}
Let $\left[\begin{array}{cccccc}(\mathbf u_I^{(2)})^T& (\xi_I^{(2)})^T& (p_I^{(2)})^T&(\mathbf u_\Gamma^{(2)})^T& (\xi_\Gamma^{(2)})^T& (p_\Gamma^{(2)})^T \end{array}\right]^T$ be the solution of the fully assembled system:
\begin{equation}\label{ineq:lowerbound2}
\mathcal A_h
\left[\begin{array}{c}\mathbf u_I^{(2)}\\ \xi_I^{(2)}\\ p_I^{(2)}\\\mathbf u_\Gamma^{(2)}\\ \xi_\Gamma^{(2)}\\ p_\Gamma^{(2)}\end{array}\right]=
\left[\begin{array}{c}-A_{II}\mathbf u_I^{(1)}-A_{I\Delta}\mathbf u_\Delta^{(1)}\\  -B_{\Gamma I}\mathbf u_I^{(1)}-B_{\Gamma\Delta}\mathbf u_\Delta^{(1)}\\
0 \\ -A_{\Gamma I}\mathbf u_I^{(1)}-A_{\Gamma\Delta}\mathbf u_\Delta^{(1)}\\
g_{\bm p_\Gamma}-\left[\begin{array}{c}B_{\Gamma I}\mathbf u_I^{(1)}+B_{\Gamma\Delta}\mathbf u_\Delta^{(1)}\\0\end{array}\right]
\end{array}\right].
\end{equation}
Then we define $\mathbf u^{(2)}=\left[\begin{array}{cc}(\mathbf u_I^{(2)})^T&(\mathbf u_\Gamma^{(2)})^T\end{array}\right]^T$.
According to the definition of $M_\lambda^{-1}$, it is elementary to show that
\begin{equation}\label{eq:lowerbound3}
\begin{split}
g_\lambda^TM_\lambda^{-1}g_\lambda =g_\lambda^TB_{\Delta,D}H_\Delta B_{\Delta,D}^Tg_\lambda=
    |\mathbf u_\Delta^{(1)}|_{H_\Delta}^2=(\mathbf u^{(1)})^TA\mathbf u^{(1)}.
  \end{split}
\end{equation}
Therefore, \eqref{eq:lowerbound3} and the definition of $M^{-1}$, imply that
\begin{equation}\label{ineq:lowerbound4}
  \bm y^T M^{-1}\bm y=g_{\bm p_\Gamma}^T M_{\xi,p,\Gamma}^{-1}g_{\bm p_\Gamma}+(\mathbf u^{(1)})^TA \mathbf u^{(1)}.
\end{equation}
Next, we use Lemma \ref{lem:apriori}, Lemma \ref{lem:SPD2pressures} and \eqref{ineq:upperbound6}  to deduce that
\begin{equation}\label{ineq:lowerbound5}
  \begin{split}
    &\mathbf u^{(2)^T}A\mathbf u^{(2)}+\eta^{(2)^T}C\eta^{(2)}+p^{(2)^T}Ep^{(2)} \\
    \leq\,& C_1\Bigg( (A\mathbf u^{(1)})^TA^{-1}(A\mathbf u^{(1)})
      +(\widetilde B\mathbf u^{(1)})^T\left[
      \begin{array}{cc}
       \!\!\! \lambda C \!\!\!&  \\
         &\!\!\! E \!\!\!
      \end{array}\right]^{-1}\widetilde \!\!\!\!\!\!B\mathbf u^{(1)} 
      +\left[\begin{array}{cc}
      \!\!\!0& \!\! g_{\bm p_\Gamma}^T\!\!\! \end{array}\right] \left[
      \begin{array}{cc}
       \!\!\! \lambda C \!\!\!&  \\
         & \!\!\! E \!\!\!
\end{array}\right]^{-1}\left[\begin{array}{c}\!\!\!0\\\!\!\! g_{\bm p_\Gamma} \!\! \end{array}\right]
        \Bigg) \\
        \leq\,& C_1(\mathbf u^{(1)}A\mathbf u^{(1)}+g_{\bm p_\Gamma}^T M_{\xi,p,\Gamma}^{-1}g_{\bm p_\Gamma})\leq C\bm y^T M^{-1}\bm y,
  \end{split}
\end{equation}
where in the second inequality we have used the same technique used in \eqref{ineq:upperbound3} to reproduce the Schur complement $M_{\xi,p,\Gamma}^{-1}$.
From \eqref{ineq:lowerbound4}, we find
\begin{equation}\label{ineq:lowerbound6}
  \bm y^T M^{-1}\bm y\geq (\mathbf u^{(1)})^TA\mathbf u^{(1)}.
\end{equation}
Consequently, noting that $\mathbf u=\mathbf u^{(1)}+\mathbf u^{(2)}$, $\eta=\eta^{(1)}+\eta^{(2)}$ and $p=p^{(1)}+p^{(2)}$, inequalities \eqref{ineq:lowerbound5} and \eqref{ineq:lowerbound6} lead us to
\[
 C_1(\mathbf u^TA\mathbf u+\eta^TC\eta+p^TEp) \leq\bm x^TGM^{-1}G\bm x ,
\]
which is the desired result, and the proof is completed.
\end{proof}

Finally, we present our main result as the following Theorem.
\begin{theorem}\label{thm:main}
 There exists a constant $C_0$ independent of $N, h, H$ such that
\begin{equation}\label{ineq:main}
    \cond(M^{-1}G)\leq C_0\left(1+\log\frac Hh\right)^2.
  \end{equation}
\end{theorem}

\section{Numerical results}\label{sec:num}
We now test the convergence rate of our dual-primal preconditioner for Biot’s consolidation model \eqref{sec:model}, also in the incompressible limit, on the domain $\Omega = [0,1]^d$.
Using the positive-definite reformulation \eqref{eq:SchurComplement}, we apply the preconditioned conjugate gradient (PCG) method using our block BDDC/FETI–DP preconditioner \eqref{sec:Preconditioner}. PCG method is initialized with a zero vector and terminates when the Euclidean norm of the residual is reduced by a factor of $10^{-8}$. The results are reported as the total number of iterations (iter), along with the minimum and maximum eigenvalues $(\eig_{\min}$ and $\eig_{\max})$ of the iteration matrix $M^{-1}G$. The 2D tests were performed using MATLAB, while the 3D parallel tests used the PETSc library \cite{petsc}.
\subsection{Scalability and quasi-optimality tests}
In 2D, we employ two types of mixed finite element pairs. The first type uses continuous finite element pairs:
$(P_1 \text{ iso } P_2)-P_1-P_1$ (cf. \cite[Section VI.6]{BF1991}), which are used to discretize the displacement, total pressure, and pressure variables, respectively. The second type employs $(P_1 \text{ iso } P_2)-P_0-P_1$ pairs, where a discontinuous element is used for the total pressure variable. In this case, the unknown $\xi_\Gamma$ does not appear in the reduced system \eqref{eq:SchurComplement}. In 3D, we use $Q_2-Q_1-Q_1$ Taylor-Hood pairs.

In the following examples, we apply Neumann boundary conditions at $x=0$ and zero Dirichlet boundary conditions on the rest of the boundary. We fix the parameter $E=10^6,\,\nu=0.499$ to ensure that all tests are conducted in the almost incompressible regime.  By observing Tables \ref{2D:continuous}-\ref{3D}, it is evident that the minimum eigenvalue remains independent of the mesh size in both cases. For a fixed $H/h$, the maximum eigenvalue does not depend on the number of subdomains. However, when the number of subdomains is fixed, $\eig_{\max}$ varies with $H/h$. Its least squares approximation is given by $C_1 + C_2(1 + \log(H/h))^2$, which is consistent with Theorem \ref{thm:main}, as shown in Figures \ref{fig}. Additionally, the convergence rate of the algorithm improves when enriched by incorporating edge-averaged functions along with the subdomain vertex nodal basis functions for each component.

\begin{table}[tb]
	\centering
	\caption{2D tests:$M^{-1}G$ minimum and maximum eigenvalue, iteration counts, continuous total pressure.}\label{2D:continuous}
	\begin{tabular}{|c|c|ccc|ccc|}
		\hline
		\multirow{2}*{$\frac{H}{h}$} & \multirow{2}*{$\#sub$}
		& \multicolumn{3}{c|}{Vertex}
		& \multicolumn{3}{c|}{Vertex + Edge averages} \\
		\cline{3-5} \cline{6-8}
		& & $\eig_{\min}$ & $\eig_{\max}$ & iter
		& $\eig_{\min}$ & $\eig_{\max}$ & iter \\
		\hline
		\multirow{5}*{12}
		& $16^2$ & 0.1999 & 4.0134 & 28
			& 0.2141 & 2.5271 & 22 \\
		& $24^2$ & 0.1962 & 4.1000 & 31
			& 0.2103 & 2.5738 & 24 \\
		& $32^2$ & 0.1942 & 4.1441 & 31
			& 0.2083 & 2.5979 & 24 \\
		& $40^2$ & 0.1930 & 4.1711 & 32
			& 0.2071 & 2.6128 & 24 \\
		& $48^2$ & 0.1921 & 4.1868 & 32
			& 0.2063 & 2.6230 & 24 \\
		\hline
		\multirow{2}*{$\#sub$} & \multirow{2}*{$\frac{H}{h}$}
		& \multicolumn{3}{c|}{Vertex}
		& \multicolumn{3}{c|}{Vertex + Edge averages} \\
		\cline{3-5} \cline{6-8}
		& & $\eig_{\min}$ & $\eig_{\max}$ & iter
		& $\eig_{\min}$ & $\eig_{\max}$ & iter \\
		\hline
		 \multirow{5}*{$10^2$} & $16$ & 0.2186 & 4.2919 & 28
			& 0.2293 & 2.7836 & 23 \\
		&$24$  & 0.2347 & 4.9564 & 30
			& 0.2413 & 3.3073 & 25 \\
		&$32$                     & 0.2448 & 5.4588 & 31
			& 0.2495 & 3.7089 & 25 \\
		&$40$                        & 0.2520 & 5.8660 & 33
			& 0.2554 & 4.0374 & 26 \\
		&$48$                       & 0.2573 & 6.2101 & 33
			& 0.2600 & 4.3170 & 28 \\
		\hline
	\end{tabular}
\end{table}
\begin{table}[tb]
	\centering
	\caption{2D tests: $M^{-1}G$ minimum and maximum eigenvalue, iteration counts, discont. total pressure.}\label{2D:discontinuous}
	\begin{tabular}{|c|c|ccc|ccc|}
		\hline
		\multirow{2}*{$\frac{H}{h}$} & \multirow{2}*{$\#sub$}
		& \multicolumn{3}{c|}{Vertex}
		& \multicolumn{3}{c|}{Vertex + Edge averages} \\
		\cline{3-8}
		&
		& $\eig_{\min}$ & $\eig_{\max}$ & iter
		& $\eig_{\min}$ & $\eig_{\max}$ & iter \\
		\hline
		\multirow{5}{*}{12}
		& $16^2$ & 0.2911 & 3.6703 & 22
			& 0.3233 & 2.3328 & 19 \\
		& $24^2$ & 0.2854 & 3.7434 & 24
			& 0.3174 & 2.3768 & 20 \\
		& $32^2$ & 0.2825 & 3.7755 & 25
			& 0.3174 & 2.3923 & 20 \\
		& $40^2$ & 0.2870 & 3.7791 & 25
			& 0.3144 & 2.3996 & 20 \\
		& $48^2$ & 0.2795 & 3.7764 & 25
			& 0.3125 & 2.4036 & 20 \\
		\hline\multirow{2}*{$\#sub$} &
		\multirow{2}*{$\frac{H}{h}$}
		& \multicolumn{3}{c|}{Vertex}
		& \multicolumn{3}{c|}{Vertex + Edge averages} \\
		\cline{3-8}
		&
		& $\eig_{\min}$ & $\eig_{\max}$ & iter
		& $\eig_{\min}$ & $\eig_{\max}$ & iter \\
		\hline
		\multirow{5}{*}{$10^2$}
		& $16$   & 0.2993 & 3.9263 & 22
			& 0.3188 & 2.5619 & 19 \\
		& $24$   & 0.2982 & 4.5301 & 24
			& 0.3067 & 3.0410 & 22 \\
		& $32$   & 0.2978 & 4.9899 & 26
			& 0.3018 & 3.4113 & 23 \\
		& $40$   & 0.2977 & 5.3643 & 27
			& 0.2994 & 3.7158 & 24 \\
		& $48$   & 0.2976 & 5.6815 & 28
			& 0.2979 & 3.9758 & 25 \\
		\hline
	\end{tabular}
\end{table}

\begin{table}[tb]
    \centering
   \vspace{-0.4cm} \caption{\text{3D} tests: $M^{-1}G$ minimum and maximum eigenvalue, iteration counts,  $E=10^6,\,\nu=0.499$. Left: fixed $4\times 4\times 4$ subdomains. Right: fixed $H/h=8$}\label{3D}
    \begin{minipage}{0.45\textwidth}
        \centering
        \begin{tabular}{|lccc|}
            \hline
            $H/h$ & $\eig_{\min}$ & $\eig_{\max}$ & iter \\
            \hline
            1 & 0.2902 & 1.7560 & 20 \\
            2 & 0.2342 & 6.1108 & 40 \\
            4 & 0.2312 & 9.4024 & 51 \\
            8 & 0.2318 & 13.8449 & 62 \\
            \hline
        \end{tabular}
    \end{minipage}\hfill
    \begin{minipage}{0.45\textwidth}
        \centering
        \begin{tabular}{|lccc|}
            \hline
            $\#sub$ & $\eig_{\min}$ & $\eig_{\max}$ & iter \\
            \hline
            $3\times 3\times 3$ & 0.2322 & 14.4066 & 59 \\
            $4\times 4\times 4$ & 0.2318 & 13.8449 & 62 \\
            $5\times 5\times 5$ & 0.2343 & 13.6503 & 62 \\
            $6\times 6\times 6$ & 0.2340 & 13.5673 & 64 \\
            $7\times 7\times 7$ & 0.2635 & 13.4782 & 61 \\
            $8\times 8\times 8$ & 0.2338 & 13.5287 & 64 \\
            \hline
        \end{tabular}
    \end{minipage}
\end{table}

\begin{figure}[tb]
\centering
\subfigure 
{\begin{minipage}[t]{6cm}
\centering
\includegraphics[width=5cm, height=3.5cm]{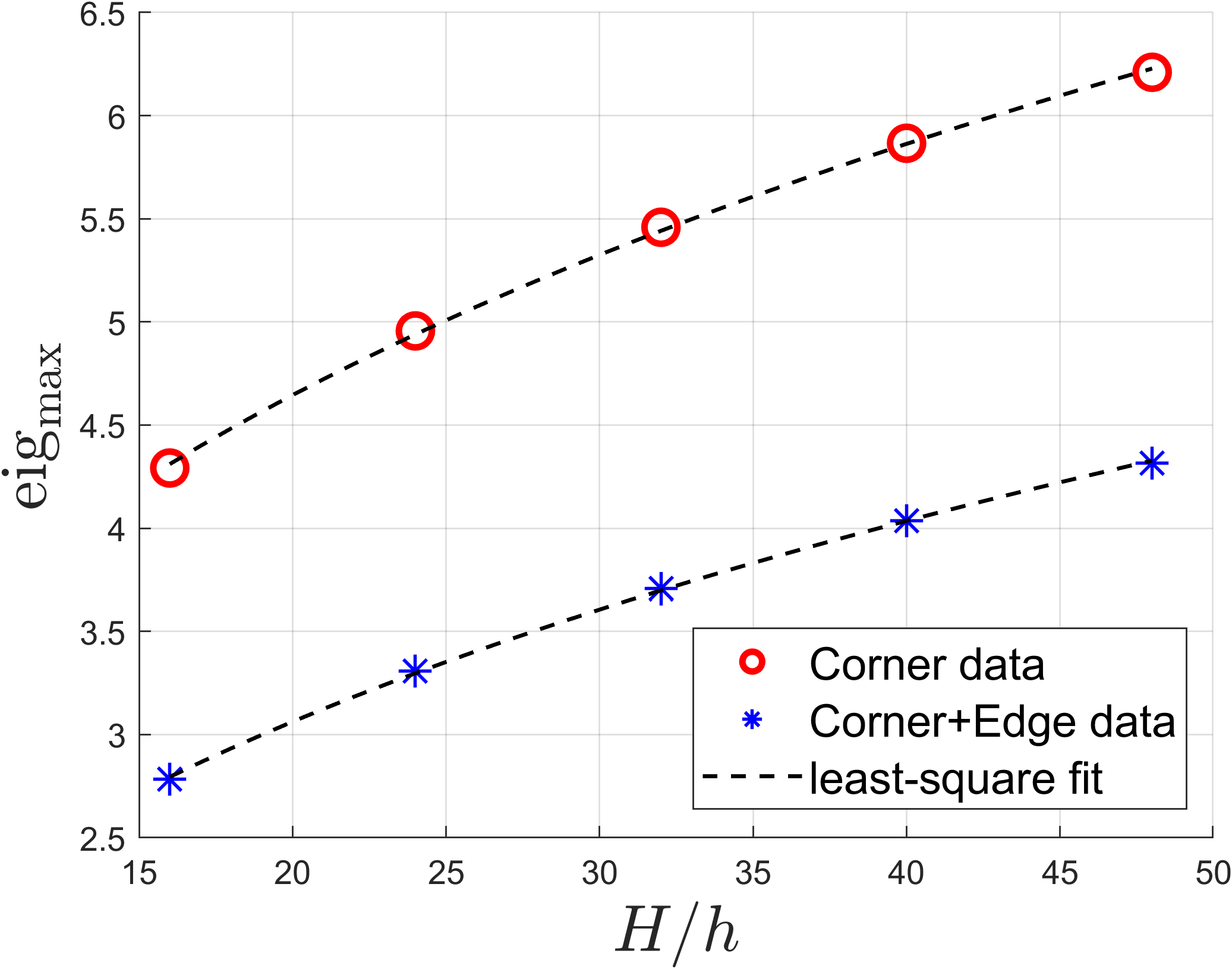}
\end{minipage}
}
\hspace{3mm}\subfigure 
{\begin{minipage}[t]{6cm}
\centering
\includegraphics[width=5cm, height=3.5cm]{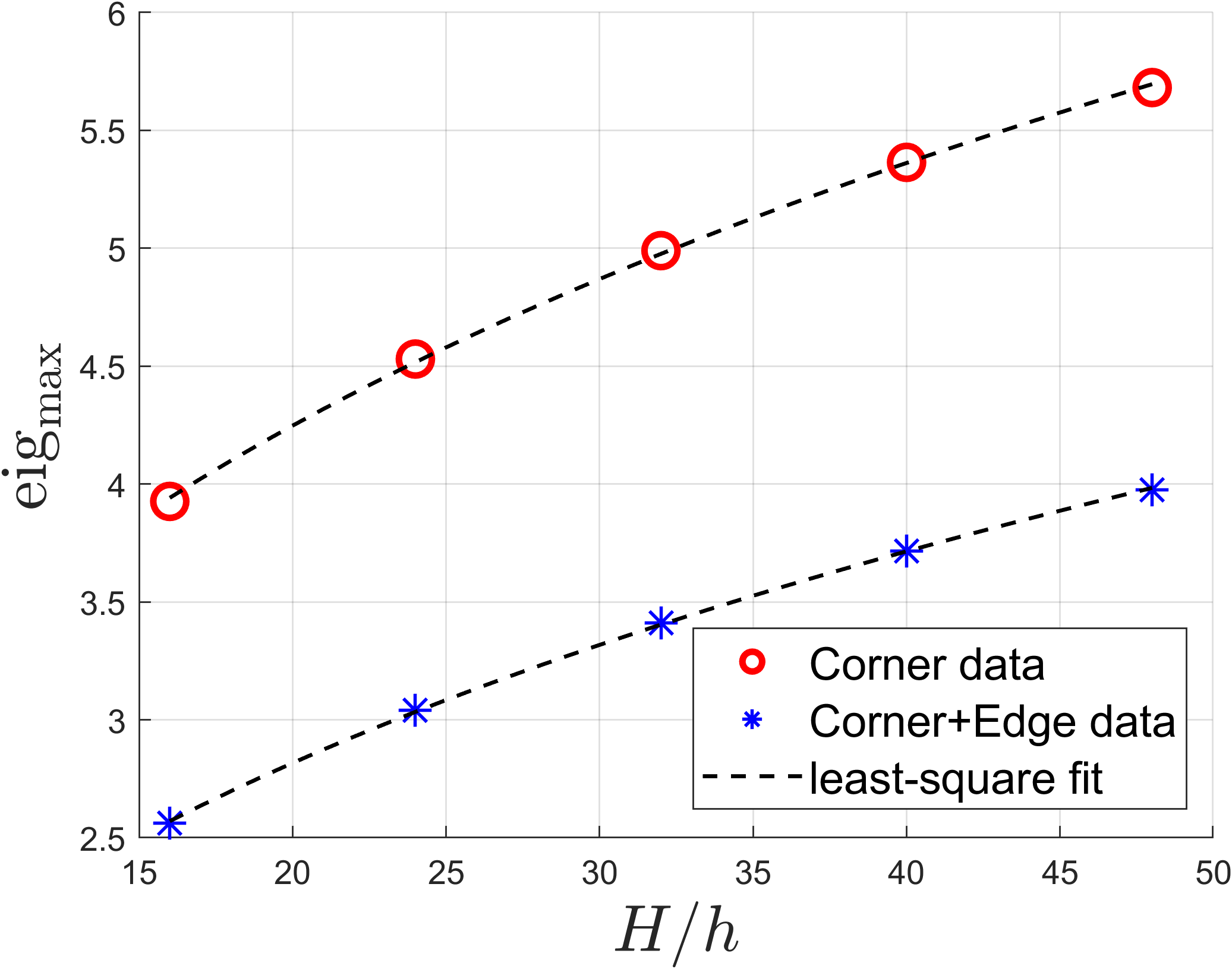}
\end{minipage}
}
\vspace{-0.4cm}\caption{2D tests: least squares fits of $\eig_{\max}$ from Tables \ref{2D:continuous}-\ref{2D:discontinuous} with $C_1 + C_2(1 + \log H/h)^2$ plots. Left: continuous total pressure. Right: discontinuous total pressure.}\label{fig}
\end{figure}

\subsection{Jumping coeﬃcients test}
We next study the performance of our BDDC/ FETI-DP block preconditioner when parameters exhibit discontinuities across subdomain interfaces, using a continuous pressure space with a checkerboard pattern. In 2D, the computational domain is discretized using $H=1/12,\,h=1/192$ and subdivided into $N = 12 \times 12$ subdomains. In Table \ref{2D:parameter} (top-left), the Poisson’s ratio is fixed at $\nu = 0.49 $ throughout the domain, while Young’s modulus $E$ varies in the black subdomains, with $E=1$ in the white subdomains.
In Table \ref{2D:parameter} (top-right), Young’s modulus is fixed at $E = 10^6$ everywhere, while the Poisson’s ratio $\nu$ varies in the black subdomains, with $\nu = 0.3$ in the white subdomains.   In Table \ref{2D:parameter} (bottom-left), we fix the hydraulic conductivity $\kappa=1$ and vary the Biot-Willis constant $\alpha$ in the black subdomains, while keeping $\alpha=1$ in the white subdomains.
In Table \ref{2D:parameter} (bottom-right), we fix Biot-Willis constant $\alpha=1$. The hydraulic conductivity $\kappa$ varies in the black subdomains, while it remains $\kappa=1$ in the white subdomains. These results show that our algorithm remains robust even in the presence of jumping coefficients.

\begin{table}[tb]
    \centering
    \caption{2D tests with jumping coefficients: $H=1/12,\,h=1/192$ and continuous total pressure.}
    \label{2D:parameter}
    \begin{tabular}{|lccc|lccc|}
        \hline
        $E$ & $\eig_{\min}$ & $\eig_{\max}$ & iter & $\nu$ & $\eig_{\min}$ & $\eig_{\max}$ & iter \\
        \hline
        1 & 0.2375 & 4.3772 & 32 & 0.2 &0.8646&4.5264  & 19 \\
        10 & 0.1232 & 2.8534 & 37 & 0.3 & 0.6824 & 4.4671 & 20 \\
        1000 & 0.0321 & 2.5417 & 58 & 0.45 & 0.5025 & 4.4230 & 23 \\
        100000 & 0.0297 & 2.5391 & 53 & 0.499 & 0.4308 & 4.4006 & 23 \\
        10000000 & 0.0296 & 2.5391 & 53 & 0.49999 & 0.4292 & 4.4002 & 23 \\
        \hline
        $\alpha$ & $\eig_{\min}$ & $\eig_{\max}$ & iter & $\kappa$ & $\eig_{\min}$ & $\eig_{\max}$ & iter \\
        \hline
        $10^{-2}$ & 0.2376 & 4.3773 & 29 & $10^{-1}$ & 0.2376 & 4.3773 & 29 \\
        $10^{-4}$ & 0.2376 & 4.3773 & 29 & $10^{-3}$ & 0.2376 & 4.3773 & 29 \\
        $10^{-6}$ & 0.2376 & 4.3773 & 29 & $10^{-5}$ & 0.2376 & 4.3773 & 30 \\
        $10^{-8}$ & 0.2376 & 4.3773 & 29 & $10^{-7}$ & 0.2372 & 4.3773 & 38 \\
        $10^{-10}$ & 0.2376 & 4.3773 & 29 & $10^{-9}$ & 0.2294 & 4.3773 & 41 \\
        \hline
    \end{tabular}
\end{table}

\subsection{Stability test of displacement in $H_0^1$ space for almost incompressible systems}

In this test, we primarily test the stability of the preconditioned system in the nearly incompressible case, where all boundaries are zero Dirichlet boundary conditions. We continue to use the two types of mixed finite element spaces from the first example in two dimensions. Table \ref{valid:eig} shows that as the material becomes increasingly incompressible when the Lam\'e parameter $\lambda$ tends to infinity, a single smallest eigenvalue approaches zero. Here, we select the second smallest eigenvalue as the valid minimum eigenvalue, which is reasonable, as discussed in \cite{LMW2017,NR2006}. Indeed, the iteration counts show only a modest increase when $\nu \rightarrow 0.5$, confirming the robustness of our algorithm.

\begin{table}[tb]
    \centering
    \caption{2D tests: BDDC/FETI-DP minimum eigenvalue, valid minimum eigenvalue, maximum eigenvalue, and iteration counts, $h=1/144,\, E=10^6,\,\alpha=\kappa=1$.}
    \label{valid:eig}
    \resizebox{\textwidth}{!}{
    \begin{tabular}{|llccc|lccc|}
        \hline
        \multicolumn{5}{|c|}{Continuous total pressure} & \multicolumn{4}{c|}{Discontinuous total pressure} \\
        \cline{2-9}
        $\nu$ & $\eig_{\min}$ & $\text{valid} \,\eig_{\min}$ & $\eig_{\max}$ & iter & $\eig_{\min}$ & $\text{valid} \,\eig_{\min}$ & $\eig_{\max}$ & iter \\
        \hline
        0.3 & 0.53791 & 0.53791 & 3.72586 & 18 & 0.59760 & 0.59760 & 3.56179 & 17 \\
        0.4 & 0.28315 & 0.28315 & 3.72586 & 22 & 0.35308 & 0.35308 & 3.45150 & 18 \\
        0.45 & 0.14811 & 0.14811 & 3.69605 & 25 & 0.19425 & 0.19425 & 3.38803 & 19 \\
        0.49 & 0.03130 & 0.32542 & 3.69243 & 29 & 0.042283 & 0.51566 & 3.33259 & 22 \\
        0.499 & 0.003178 & 0.30733 & 3.68831 & 33 & 0.004314 & 0.49255 & 3.31948 & 25 \\
        0.4999 & 0.000318 & 0.30553 & 3.68853 & 36 & 0.000432 & 0.49018 & 3.31816 & 27 \\
        0.49999 & 0.0000318 & 0.30535 & 3.68924 & 40 & 0.0000432 & 0.48994 & 3.31803 & 29 \\
        \hline
    \end{tabular}}
\end{table}

\section{Conclusion}

This paper develops a block BDDC/FETI-DP preconditioner for Biot’s consolidation model using three-field mixed finite elements with displacement, pressure, and total pressure variables. The domain is divided into nonoverlapping subdomains, and displacement continuity is enforced by Lagrange multipliers. After eliminating displacement and interior pressure variables, the problem is reduced to a symmetric positive definite linear system for interface pressure, total pressure, and the Lagrange multiplier variables, solved using a preconditioned conjugate gradient method.
The block preconditioner combines BDDC preconditioners for the interface pressure and total pressure blocks with the FETI-DP preconditioner for the Lagrange multiplier block. The condition number of the resulting preconditioned Biot's operator is shown to be scalable in the number of subdomains, polylogarithmic in the ratio of subdomain and mesh sizes, and robust with respect to the model parameters.

\bibliographystyle{siamplain}
\bibliography{mybibfile}

\end{document}